\documentclass[12pt]{article} %
\usepackage{fullpage}
\usepackage{amsthm}
\usepackage{amsmath, amssymb, graphicx, url}

\newcommand{\N}{\mathbb{N}}

\newcommand{\R}{\mathbb{R}}

\newcommand{\cA}{\mathcal{A}}

\newcommand{\cN}{\mathcal{N}}
\newcommand{\cP}{\mathcal{P}}
\newcommand{\cV}{\ell}
\newcommand{\cX}{\mathcal{X}}

\newcommand{\cO}{\mathcal{O}}
\newcommand{\cS}{\mathcal{S}}

\newcommand{\eps}{\varepsilon}

\usepackage{enumitem}
\usepackage[normalem]{ulem}

\newcommand{\dd}{{\rm d}}

\DeclareMathOperator{\Tr}{Tr}
\DeclareMathOperator{\KL}{KL}

\def\KL#1#2{\textnormal{KL}({#1}\Vert{#2})}

\usepackage{mathtools}
\DeclarePairedDelimiter\ceil{\lceil}{\rceil}

\usepackage{algorithm,algpseudocode} 
\usepackage{authblk}

\title{Multilevel Optimization for Inverse Problems\footnote{This work
    was finished during the visit of Simon Weissmann and Jakob Zech at
    MIT. SW and JZ would like to thank Youssef Marzouk and Ashia
    Wilson for the invitation to visit, for providing a fruitful
    research environment and for their kind hospitality. Moreover, SW and
    JZ thank Universit\"at Heidelberg for helping to fund the visit as
    part of the program ``Mobilit\"atsmassnahmen im Rahmen
    internationaler Forschungskooperationen".}}

\author[1]{Simon Weissmann}
\author[2]{Ashia Wilson}
\author[1]{Jakob Zech}
\date{\today}
\affil[1]{\normalsize
  Universit\"at Heidelberg, Interdisziplin\"ares Zentrum f\"ur Wissenschaftliches Rechnen, D-69120 Heidelberg, Germany\\
  
\texttt{simon.weissmann@uni-heidelberg.de}, \texttt{jakob.zech@uni-heidelberg.de}
}
\affil[2]{Massachusetts Institute of Technology, Department of Electrical Engineering and Computer Sciences, Cambridge, MA, 02139, USA\\

  \texttt{ashia07@mit.edu}
}

\newtheorem{theorem}{Theorem}[section]

\newtheorem{corollary}[theorem]{Corollary}

\newtheorem{definition}[theorem]{Definition}
\newtheorem{example}[theorem]{Example}

\newtheorem{lemma}[theorem]{Lemma}

\newtheorem{proposition}[theorem]{Proposition}
\newtheorem{remark}[theorem]{Remark}

\newtheorem{assumption}[theorem]{Assumption}

\begin{document}

\maketitle

\begin{abstract}
  Inverse problems occur in a variety of parameter identification
  tasks in engineering. Such problems are challenging in practice, as
  they require repeated evaluation of computationally expensive
  forward models.  We introduce a unifying framework of multilevel
  optimization that can be applied to a wide range of
  optimization-based solvers.  Our framework provably reduces the
  computational cost associated with evaluating the expensive forward
  maps stemming from various physical models.  To demonstrate the
  versatility of our analysis, we discuss its implications for various
  methodologies including multilevel (accelerated,  stochastic) gradient descent, a multilevel
  ensemble Kalman inversion and a multilevel Langevin sampler. We also
  provide %
  numerical experiments to verify our theoretical findings. %
\end{abstract}

% \begin{keywords}%
{\bf keywords}
multilevel methods, optimization, inverse problems
% \end{keywords}

\section{Introduction}\label{sec:Introduction}
Inverse problems are ubiquitous in %
applied mathematics and modern machine learning.  The aim %
is usually to quantify information about unknown parameters which are
indirectly observed through a noisy observation model. %
Solutions for inverse problems are often found using optimization and
sampling methods and crucially depend on an underlying physical model
incorporated through a forward map. The physical models %
are typically highly complex such that associated numerical
approximations come with extensive computational costs. The multilevel
Monte Carlo method (MLMC) \cite{MR2436856,Heinrich2001} is a
well-established variance reduction method, %
which addresses this issue by shifting a %
large part of the work to less accurate model evaluations.  In the
context of Bayesian inference, MLMC methods have been applied to
Markov chain Monte Carlo (MCMC) methods \cite{MLMCMC} as well as to
deterministic quadrature rules such as sparse grid
\cite{MR3502561,MLGPC} and quasi-Monte Carlo methods
\cite{MR2648461,MR3636617}.

In this work, we apply similar ideas to the following general
optimization problem
\begin{equation}\label{eq:general_opti}
  \min_{x\in\mathcal X}\ \Phi(x),
\end{equation}
where $\mathcal X$ is a Hilbert space and
$\Phi:\mathcal X\to \mathbb R_+$ an objective. %
The idea of multilevel optimization %
is to replace the evaluation of $\Phi$ (or its derivatives) by %
some %
approximation that becomes increasingly more accurate as the
optimization process converges. Intuitively, when the current state
may be far from minimum, it suffices to roughly move in the direction
of the minimizer; however, as the state approaches the minimizer of
$\Phi$, higher %
accuracy is required to reduce numerical bias. Multilevel optimization
strategies are targeting efficient algorithms with the aim of reducing
overall computational costs.

Such ideas have recently been applied in different contexts.  The
works closest to ours are \cite{MR4284423,MR4294188} which use
multilevel optimization on an optimal control problem, and
\cite{2104.01945}, where a multilevel version of the Stein variational
gradient descent method is introduced. The aim of our manuscript is to
formulate a unifying multilevel framework which can be a applied to a
wide range of optimization and sampling methods with particular focus
on inverse problems.

\paragraph{Contributions} Our principal contributions are three-fold:

\begin{itemize}
\item We formulate a multilevel strategy for general iterative
  optimization methods where each update step depends on an accuracy
  level. We derive an optimal choice of levels that minimizes
  computational costs %
  while ensuring to achieve a certain tolerance for the error.
  Compared to the single-level framework, we prove that the
  computational cost can be reduced by a $\log$-factor, and we provide
  an example to show that our results are sharp.
\item We use our framework to %
  introduce a multilevel ensemble Kalman inversion method and its
  extension to Tikhonov regularization. For linear forward models and
  with the incorporation of variance inflation, we prove convergence
  rates that reduce the computational costs by the expected
  $\log$-factor when compared to single-level methods.
\item We apply our framework to particle based sampling methods for
  Bayesian inference. We develop a multilevel formulation of
  interacting Langevin samplers. Viewing Langevin dynamics as gradient
  flow in the space of probability measures w.r.t. the
  ~Kullback-Leibler divergence, under certain assumptions we show
  convergence for the mean-field limit and provide a cost analysis
  that again reduces cost by a $\log$-factor compared to the
  single-level method.
\end{itemize}

\paragraph{Outline} $\S$\ref{sec:IP} discusses optimization-based
approaches for solving inverse problems while $\S$\ref{sec:ML_opti}
presents our unified multilevel optimization framework. In
$\S$\ref{sec:MLEKI} and $\S$\ref{sec:MLsampling} we apply our
framework to particle based optimization and Bayesian inference
respectively, and $\S$\ref{sec:numerics} presents numerical
experiments for these examples.

\paragraph{Notation} $f \lesssim g$ indicates the existence of $C$
such that $f(x)\le C g(x)$, with $C$ independent of $x$ in a certain
range that will be clear from context. Moreover $f\simeq g$ iff
$f\lesssim g$ and $g\lesssim f$.%
\section{Inverse Problems}\label{sec:IP}
Let $\cX$ be a Hilbert space, $n_y\in \N$ and
$F:\cX\to\R^{n_y}$ %
the so-called \emph{forward model}. We consider the task of recovering
the unknown quantity $x\in\cX$ from a (noisy) observation
$y\in\R^{n_y}$ of $F(x)$. Throughout we assume an additive Gaussian
noise model, i.e.\ $y$ is a realization of the random variable
\begin{equation}\label{eq:ip}
  Y = F(x) + \eta,
\end{equation}
with $\eta\sim\cN(0,\Gamma)$ Gaussian for a symmetric positive
definite (SPD) covariance matrix $\Gamma\in\mathbb R^{n_y\times n_y}$.

This problem is typically ill-posed in the sense of \cite{H1902}, %
for instance because the dimension of the parameter space $\cX$ may be
much higher than the dimension $n_y$ of the observation space. We now
recall two different methodologies to deal with these difficulties,
both of which recast the problem into one of optimization.

\subsection{Regularized optimization}\label{sec:regopt}
One classical approach to approximate $x$ is to minimize the objective
\begin{align}\label{eq:regularized_lossfunction}
  \Phi(x) := %
  \cV(x,y)
  +
  R(x),\qquad
  \cV(x,y) = \frac12\|\Gamma^{-1/2}(F(x)-y)\|_{\R^{n_y}}^2,
\end{align}
where $\cV$ denotes the least-squares data misfit loss functional and
$R:\mathcal X\to\R_+$ is a regularizer.  Common choices of
regularization include Tikhonov regularization \cite{EKN1989} with
$R(x) = \frac\lambda 2\|C_0 x\|^2$ and total variation regularization
\cite{Chambolle2009AnIT,ROF1992} with
$R(x) = \frac\lambda 2\|\nabla x\|^2$. Note that prior information can
be incorporated through %
$C_0\in\mathcal L(\mathcal X,\mathcal X)$.  In the following, for
fixed $y$, we use the shorthand $\cV(x):=\cV(x,y)$. %

We continue this discussion in Sec.~\ref{sec:MLEKI} where we present a
particle based multilevel optimization method to minimize $\Phi$ in
\eqref{eq:regularized_lossfunction}.  For motivation and further
discussion of regularization methods to solve~\eqref{eq:ip}, see,
e.g., \cite{EHN96,benning_burger_2018} and references therein.

\subsection{Bayesian inference}\label{sec:baysinf}
In the Bayesian approach (e.g.~\cite{AMS10}) the parameter and
observation are modeled as a joint random variable $(X,Y)$ on
$\mathcal X\times \mathbb R^{n_y}$. The goal is to determine the
\emph{posterior}, which refers to the conditional distribution of $X$
given the realization $y$ of $Y$ in \eqref{eq:ip}.  Assume $X$ and
$\eta$ to be stochastically independent, and let $X\sim Q_0$ for a
\emph{prior distribution} $Q_0$.  {Under certain technical assumptions
  \cite[Theorem~6.31]{AMS10},} the posterior {$Q^y_\ast$} is then
well-defined, absolutely continuous with respect to the prior, and
$Q^y_\ast({\mathrm d}x) = \frac{1}{Z} \exp(-\cV(x,y))Q_0({\mathrm
  d}x)$, where
$Z=\int_{\mathcal X} \exp(-\cV(x,y))Q_0({\mathrm d}x)\in\R$ is a
normalizing constant and $\cV$ is given
by~\eqref{eq:regularized_lossfunction}.

Suppose for the moment that $\cX=\R^{n_x}$ is finite dimensional and
the prior $Q_0 = \cN(0,\frac1\lambda C_0)$ is Gaussian,
$C_0\in\R^{n_x\times n_x}$ SPD, $\lambda>0$.  Then, %
the posterior $Q^y_\ast$ has Lebesgue density
$\rho_\ast(x) = \frac{1}{Z}\exp(-\cV(x,y)+R(x)),$ with
$R(x)=\frac{\lambda}{2}\|C_0^{-1/2} x\|_{\R^{n_x}}^2.$ In the Bayesian
framework, solving the inverse problem amounts to sampling from the
posterior. One way to achieve this is by minimizing the objective
\begin{equation}\label{eq:sampling_objective}
  \Phi(\psi) = \KL{\psi}{\rho_*},
\end{equation}
for $\psi$ in a given family of (tractable) probability distributions
on $\cX$. Here ${\rm KL}$ stands for the Kullback–Leibler
divergence~\cite{Kullback}.  Hence, we end up again with an
optimization problem, but this time over a subspace of the probability
measures on $\cX$.  This discussion will be continued in
Sec.~\ref{sec:MLsampling} where we present a multilevel optimization
algorithm to minimize \eqref{eq:sampling_objective}.

While the approaches in \eqref{eq:regularized_lossfunction} and
\eqref{eq:sampling_objective} are entirely different, we emphasize
that the minimization of both objectives requires multiple evaluations
of the forward model $F$, which might be very costly in practice.

\begin{example}\label{ex:1}
  Let $\cX=L^2(D)$ for a convex bounded polygonal domain
  $D\subseteq\R^2$. By classical PDE theory, for every $f\in \cX$, the
  equation
  \begin{equation}\label{eq:PDE}
    \begin{cases}
      -\Delta u_f(s)+u_f(s) = f(s) &s\in D,\\
      u_f(s) = 0 &s\in\partial D,
    \end{cases}
  \end{equation}
  has a unique weak solution
  $u_f\in H^2(D)\cap H_0^1(D)\subseteq \cX$. %
  Let $\cO:\cX\to\R^{n_y}$ be a bounded linear map called the
  \emph{observation operator}. The forward model
  $F(f):=\cO(u_f)\in\R^{n_y}$ then ``observes'' the solution of
  \eqref{eq:PDE} through the functional $\cO$.

  Given noisy observations $y=F(f)+\eta$ as in \eqref{eq:ip}, any
  method minimizing the objectives in
  \eqref{eq:regularized_lossfunction} or \eqref{eq:sampling_objective}
  has to access $\Phi$ (or its derivatives) and thus repeatedly
  evaluate the forward model $F$. Each such evaluation requires
  solving \eqref{eq:PDE}. Since \eqref{eq:PDE} has no closed form
  solution, $u_f$ can only be \emph{approximated} using a numerical
  PDE solver such as the finite element method (FEM).
\end{example}

\section{A unified multilevel optimization framework}\label{sec:ML_opti}
In order to minimize an objective $\Phi$ as in
\eqref{eq:general_opti}, we consider an abstract optimization method
described by the fixed point iteration
\begin{equation}\label{eq:Psi}
  x_{k+1} := \Psi (x_k),\quad x_0 \in\cX.
\end{equation}
For certain applications, an exact evaluation of $\Psi$ is either not
possible, or computationally infeasible. In such cases, typically
numerical approximations $\Psi_l$ to $\Psi$ are available. Here the
``level'' $l$ is a positive real number and can be understood as the
accuracy of the approximation---the larger $l$ the closer $\Psi$ and
$\Psi_l$ are. The precise meaning of this statement will be quantified
in the following.  At the same time, higher accuracy comes at higher
computational cost, which is accounted for by the assumption that one
evaluation of $\Psi_l$ amounts to computational cost $l$.

\begin{remark}
  In practice $\Psi_l$ might only be available for certain
  $l\in\N$. For simplicity we allow $l\in\R$, $l>0$, but mention that
  our analysis extends to the discrete case by rounding $l$ to the
  next larger admissible %
  level.
\end{remark}

Replacing $\Psi$ in the update rule \eqref{eq:Psi} %
with $\Psi_l$ leads to
\begin{equation}\label{eq:disc_seq}
  x_{k+1} =  \Psi_{l_k}(x_k),\quad x_0 \in\cX.
\end{equation}
Here, $l_k\in\N$ is the level in iteration $k$ of the optimization
process. The goal is to choose levels which minimize the overall
computational cost while achieving fast convergence.

We denote the error of the $k$th iterate $x_k$ by $e_k$. For example,
if %
$x_*$ is the unique minimizer of $\Phi$, $e_k$ could stand for
$\|x_k-x_*\|_{\cX}$, or for the distance of the objective to the
minimum, i.e.\ $\Phi(x_k)-\Phi(x_*)$. Our analysis is based on the
following abstract assumption. It can be understood as a form of
linear convergence, up to an additive term stemming from the
approximation of $\Psi$ by $\Psi_l$.

  \begin{assumption}\label{ass:disc}
    There exists $c \in (0,1)$ and $\alpha>0$ such that for any choice
    of levels $l_k\ge 1$ and with $x_{k}$ %
    as in \eqref{eq:disc_seq},
    \begin{enumerate}
    \item \label{item:errdecay}
      {\bf error decay:} for all $k\in\N$
      \begin{equation}\label{eq:disc}
        e_{k+1}
        \le c %
        e_k
        + l_k^{-\alpha}, %
      \end{equation}
    \item %
      {\bf cost model:} for all $k\in\N$, the cost of computing $x_k$
      in \eqref{eq:disc_seq} equals
      \begin{align}\label{eq:cost}
        {\rm cost}(x_k)=
        \sum_{j=0}^{k-1} l_j.
      \end{align}
    \end{enumerate}
  \end{assumption}

  We motivate Assumption~\eqref{ass:disc} by verifying
  conditions~\eqref{eq:disc} and~\eqref{eq:cost} on several examples.

\begin{example}\label{ex:gdagd}
  Suppose that the objective $\Phi$ is $L$-smooth and $\mu$-strongly
  convex, and that for each $l\in\N$ we have access to functions
  $g_l:\cX\to\cX$ or random variables $G_l(x)\in\cX$ for all
  $x\in\cX$, such that for some $0\leq \eta <\infty$
  \begin{equation}\label{eq:ass}
    \|\nabla\Phi(x)-g_l(x)\|\le  \frac{l^{-\alpha}}{\eta} \quad \text{\em  or } \quad\mathbb{E} \|\nabla\Phi(x)-G_l(x)\|\le  \frac{l^{-\alpha}}{\eta} \qquad\forall x\in\cX.
  \end{equation}
  Then gradient descent using the approximate gradients, i.e.\ iterates generated by
  \begin{equation*}
    x_{k+1}=x_k-\eta_kg_{l_k}(x_k)
  \end{equation*}
  can be shown to satisfy error decay~\eqref{eq:disc}.  Interpreting
  $l_k$ as the computational cost of evaluating $g_{l_k}$, the overall
  cost to compute $x_k$ follows our cost
  model~\eqref{eq:cost}. Similarly, accelerated gradient descent and
  the stochastic versions of both algorithms with dynamically
  increasing batch sizes can be shown to satisfy \eqref{eq:disc}.
\end{example}
Details for Example~\ref{ex:gdagd} and the implications of our results
for gradient descent, accelerated gradient descent and their
stochastic versions are %
given in~Appendix \ref{app:gdagd}.

\begin{remark}
  Assumption \ref{ass:disc} states that the relation between
  computational cost and corresponding error is of the type
  ``${\rm error}\sim{\rm cost}^{-\alpha}$'' for some $\alpha>0$. To
  clarify, consider the following examples in the context of applying
  gradient descent to minimize $\Phi$:
  \begin{itemize}
    \item Fix $\gamma>0$.  Suppose we have access to an
      algorithm, that for $n\in\mathbb{N}$ requires computational cost
      $f(n):=n^\gamma$ to compute $\nabla\Phi$ up to accuracy
      $n^{-\alpha}$. With $l:=n^\gamma$, this is equivalent to saying
      that at level $l$ the error is of order $l^{-\alpha/\gamma}$,
      which fits our setting. 
      Without loss of generality we can work with
      $l$ rather than $n$.

    \item Suppose we have access to an algorithm that requires time
      $t>0$ to approximate $\nabla\Phi$ up to accuracy
      $t^{-\alpha}$. Then the level $l_k$ can be understood as the CPU
      time $t$ invested in the \emph{approximate} computation of $\nabla\Phi(x_k)$.
    \end{itemize}
  \end{remark}

Next, we %
discuss gradient descent for our running Example \ref{ex:1}. Further
details are contained in Appendix \ref{app:running}.

\begin{example}[Continuation of Example \ref{ex:1}]\label{ex:2}
  Let $F(f)=\cO(u_f)$, where $u_f$ solves \eqref{eq:PDE} and
  $\cO:L^2(D)\to\R^{n_y}$ is bounded linear, i.e.\
  $\cO(p)=\int_D\xi p$ for some $\xi\in L^2(D,\R^{n_y})$ and all
  $p\in L^2(D)$.

  Let
  $\Phi(f)=\frac12\|\Gamma^{-1/2}(F(f)-y)\|_{\R^{n_y}}^2
  +\frac{\lambda}{2}\|f\|_{L^2(D)}^2$ as in
  \eqref{eq:regularized_lossfunction}.  Then
  \begin{equation}\label{eq:nablaPhiPDE}
    \nabla\Phi(f) = u_h+\lambda f\in L^2(D),
  \end{equation}
  where $u_h$ solves \eqref{eq:PDE} with right-hand side
  $h(\cdot)=(\cO(u_f)-y)^\top\Gamma^{-1}\xi(\cdot)\in L^2(D)$.  To
  approximate $\nabla\Phi(f)$, we use linear finite elements on a
  uniform mesh on $D\subseteq\R^2$ to first obtain an approximation
  $u_f^l$ satisfying $\|u_f-u_f^l\|_{L^2(D)}\lesssim l^{-1}$, and
  subsequently with $\tilde h=(\cO(u_f^l)-y)^\top\Gamma^{-1}\xi$, a
  FEM approximation $u_{\tilde h}^l$ satisfying
  $\|u_h-u_{\tilde h}^l\|_{L^2(D)}\lesssim l^{-1}$. Here $l$
  corresponds to the dimension of the FEM space, and can thus be
  interpreted as the complexity of computing $u_{\tilde h}^l$. %
  Note, $g_l(f):=u_{\tilde h}^l+\lambda f\in L^2(D)$ yields an
  approximation to $\nabla\Phi(f)$ s.t.\
  $\|\nabla\Phi(f)-g_l(f)\|_{L^2(D)}\lesssim l^{-1}$. Hence (for fixed
  $f$ and up to a constant) $g_l(f)$ satisfies the first inequality in
  \eqref{eq:ass} with $\alpha=1$.
\end{example}

Having established %
the basic setting, we next illustrate how accuracy levels $l_j$ can be
chosen %
optimally to minimize computational costs. Recursively expanding
\eqref{eq:disc}, we get the following upper bound on the error
\begin{align}\label{eq:ek}
  e_k\le c (ce_{k-2}+l_{k-1}^{-\alpha})+l_k^{-\alpha}
  \le\cdots\le
  c^{k} e_0 + \sum_{j=0}^{k-1} c^{k-1-j} l_j^{-\alpha} =:\tilde e_k((l_j)_{j}).
\end{align}
In case $l_j=l$ for all $j=1,\dots,K-1$, we will also use the notation
$\tilde e_k(l)$.

We next determine %
levels achieving (almost) minimal cost under the constraint
$\tilde e_k%
\le\eps$.

\subsection{Single-level}\label{sec:SL}
Fix the number of iteration steps $K \in \N$. %
For the \emph{single-level method}, the level $l_j$ is fixed at a
(single) value $\bar l_K>0$ throughout the whole iteration
$j=0,\dots,K-1$. By assumption \eqref{eq:cost},
${\rm cost}(x_K) = K{\bar l_K}$.  We wish to minimize the cost under
the error constraint $\tilde e_K\le\eps$. To slightly simplify the
problem for the moment, we instead demand both terms %
in the definition of $\tilde e_k$ in \eqref{eq:ek} to be bounded by
$\frac{\eps}{2}$ (so that in particular $\tilde e_K\le \eps$). More
precisely, ${\bar l_K}>0$ should be minimal such that
\begin{align}\label{eq:lK}
  c^{K} e_0\le \frac{\eps}{2},\qquad ({\bar l_K})^{-\alpha}\frac{1-c^K}{1-c} \le \frac{\eps}{2}.
\end{align}
The first inequality implies $K\ge \frac{\log(\eps/(2e_0))}{\log(c)}$,
and the second inequality implies
\begin{align}\label{eq:lKSL}
  {\bar l_K} \geq \left(2\frac{1-c^K}{(1-c)\eps}\right)^{1/\alpha}.
\end{align}

We choose ${\bar l_K}$ so that \eqref{eq:lKSL} holds with
equality. Given $K\mapsto {\rm cost}(x_K)=Kl_{K}$ is monotonically
increasing, in order to get error $\eps$ at possibly small cost the
following choices suffice:
\begin{align}\label{eq:SL}
  {\bar l_K}(\eps) := \left(2\frac{1-\frac{\eps}{2e_0}}{(1-c)\eps}\right)^{1/\alpha},\qquad
  K(\eps) := \left\lceil\frac{\log(\eps/(2e_0))}{\log(c)}\right\rceil.
\end{align}
We next introduce a notion of optimality, and then summarize our
observations in Theorem \ref{thm:SL}.

\begin{definition}[Quasi-optimal single level choice]\label{def:qosl}
  A family %
  of reals ${\bar l_K}(\eps)>0$ and integers $K(\eps)\in\N$ %
  satisfying $\tilde e_{K(\eps)}({\bar l_K}(\eps))\le\eps$ for all
  $\eps>0$, is a \emph{quasi-optimal single level choice} iff
  \begin{equation*}
    K(\eps){\bar l_K}(\eps) =
    O\big(\inf\{\hat K\hat l\,:\,\tilde e_{\hat K}(\hat l)\le\eps\},~\hat K\in\N,~\hat l>0 \big)
    \qquad\text{as}\quad\eps\to 0.
  \end{equation*}
\end{definition}

\begin{theorem}[Single-level convergence]\label{thm:SL}
  Equation \eqref{eq:SL} defines a quasi-optimal single level choice.
  It holds
  \begin{equation}\label{eq:costSL}
    {\rm cost}_{\rm SL}(\eps):= K(\eps){\bar l_K}(\eps)\simeq
    \log(\eps^{-1})\eps^{-\frac{1}{\alpha}}
    \qquad\text{as}\quad\eps\to 0.
  \end{equation}
\end{theorem}

\begin{proof}
  For the proof see Appendix~\ref{app:proof_Prop_SL}.
\end{proof}

Due to the quasi-optimality, the (single-level-) cost behaviour
\eqref{eq:costSL} cannot be improved as $\eps\to 0$.

\subsection{Multilevel}\label{sec:ML}
Fix the number of iterations $K\in\N$.  We now allow for varying
levels throughout the optimization process. That is, we wish to find
$l_{K,j}(\eps)=l_{K,j}>0$ such that
${\rm cost}(x_K)=\sum_{j=0}^{K-1} l_{K,j}$ is minimized under the
constraint of both terms in \eqref{eq:ek} being bounded by
$\frac{\eps}{2}$, i.e.\ such that
\begin{align}\label{eq:Lkj}
  c^{K} e_0\le\frac{\eps}{2},\qquad \sum_{j=0}^{K-1} c^{K-1-j} l_{K,j}^{-\alpha}\le\frac{\eps}{2}.
\end{align}

The first condition gives again a lower bound on the number of
iterations $K$ as in Sec.~\ref{sec:SL}. Minimizing ${\rm cost}(x_k)$
under the second condition gives:
\begin{lemma}\label{lemma:ML}
  For every $K\in\N$, $\eps>0$
  \begin{align}\label{eq:optimal_level}
    l_{K,j}(\varepsilon) = C_{K,\eps}\cdot c^{\frac{K-1-j}{1+\alpha}},\quad C_{K,\eps} = \left(\frac{\varepsilon}2\right)^{-\frac1\alpha} \left(
    \frac{1-c^{\frac{K}{1+\alpha}}}{1-c^{\frac{1}{1+\alpha}}}
    \right)^{\frac1\alpha}
  \end{align}
  minimizes $\sum_{j=0}^{K-1}l_{K,j}$ under the constraint
  $\sum_{j=0}^{K-1} c^{K-1-j} l_{K,j}^{-\alpha}\le\frac{\eps}{2}$.
\end{lemma}
\begin{proof}
  For the proof see Appendix~\ref{app:proof_Lemma_ML}.
\end{proof}

Let us compute the cost. Since
$\sum_{j=0}^{K-1} \delta^{K-1-i}=\frac{1-\delta^{K}}{1-\delta}$ for
$\delta=c^{\frac{1}{1+\alpha}}\in (0,1)$,
\begin{align}\label{eq:costML}
  \sum_{j=0}^{K-1}l_{K,j}
  = C_{K,\eps} \sum_{j=0}^{K-1} c^{\frac{K-1-j}{1+\alpha}}
  =\left(\frac{\eps}{2}\right)^{-\frac{1}{\alpha}}
  \left(\frac{1-c^{\frac{K}{1+\alpha}}}{1-c^{\frac{1}{1+\alpha}}} \right)^{\frac{1+\alpha}{\alpha}}.
\end{align}
As in the single-level case, this term increases in $K$ (although it
remains bounded as $K\to\infty$). To keep the cost minimal, we choose
$K$ minimal under the first constraint in \eqref{eq:Lkj}, which leads
to %
\begin{align}\label{eq:ML}
  K(\eps) = \left\lceil\frac{\log(\eps/(2e_0))}{\log(c)}\right\rceil,\qquad
  l_{K,j}(\eps) = 
  \left(\frac{\eps}{2}\right)^{-\frac{1}{\alpha}}
  c^{\frac{K(\eps)}{1+\alpha}}
  c^{-\frac{1+j}{1+\alpha}}
  \Bigg(\frac{1-c^{\frac{K(\eps)}{1+\alpha}}}{1-c^{\frac{1}{1+\alpha}}} \Bigg)^{\frac{1+\alpha}{\alpha}}\qquad\forall j<K(\eps).
\end{align}
Observing that $c^{\frac{K(\eps)}{1+\alpha}}$ behaves like
$\eps^{\frac{1}{1+\alpha}}$, and $1-c^{\frac{K(\eps)}{1+\alpha}}\to 1$
as $\eps\to 0$, we find
$l_{K,j}(\eps)\simeq
\eps^{-\frac{1}{\alpha(1+\alpha)}}c^{-\frac{1+j}{1+\alpha}}$, with
lower and upper bounds independent of $\eps$, $K$ and $j$.

\begin{definition}[Quasi-optimal multilevel choice]\label{def:qoml}
  A family $((l_{K,j}(\eps))_{j<K(\eps)})_{\eps>0}$ of sequences
  satisfying $\tilde e_{K(\eps)}((l_{K,j}(\eps))_{j<K(\eps)})\le\eps$
  for all $\eps>0$ is a \emph{quasi-optimal multilevel choice}, iff
  \begin{align*}
    \sum_{j=0}^{K(\eps)-1}l_{K,j}(\eps) =
    O\bigg(\inf\bigg\{\sum_{j=0}^{\hat K-1}\hat l_{j}\,:\,\tilde e_{\hat K}((\hat l_j)_{j<\hat K})\le\eps,~\hat K\in\N,~\hat l_j>0~\forall j<\hat K\bigg\} \bigg)
    \qquad\text{as}\quad\eps\to 0.
  \end{align*}
\end{definition}

\begin{theorem}[Multilevel convergence]\label{thm:ML}
  Equation \eqref{eq:ML} defines a quasi-optimal multilevel choice for
  $\eps\in (0,e_0)$. It holds
  \begin{align}\label{eq:costMLasymp}
    {\rm cost}_{\rm ML}(\eps):=\sum_{j=0}^{K(\eps)-1}l_{K,j}(\eps) \simeq \varepsilon^{-\frac1\alpha}\qquad\text{as}\quad\eps\to 0.
  \end{align}
\end{theorem}
\begin{proof}
  For the proof see Appendix~\ref{app:proof_Prop_ML}.
\end{proof}

Due to the quasi-optimality, the asymptotic cost behaviour
$O(\eps^{-1/\alpha})$ required to achieve error $\tilde e_K\le \eps$
cannot be improved. Comparing with the single-level method in Theorem
\ref{thm:ML}, we observe that the multilevel method decreases the
computational cost by a factor $\log(\eps^{-1})$. In practice and for
small $\eps>0$, this can amount to a significant speedup as we will
see in our numerical examples.

In Appendix \ref{app:gdagd} we provide further details of the
implications of Theorem \ref{thm:ML} for gradient descent, accelerated
gradient descent and the stochastic versions of these algorithms. As
an application we discuss a stochastic gradient descent algorithm that
uses increasing batch sizes in Example \ref{ex:dynamicsgd}.

  \begin{remark}\label{rmk:constant}
    Suppose that $e_k$ generated with levels $l_j$ satisfies instead
    of \eqref{eq:disc} the relaxed condition
    $e_{k+1}\le ce_k+Cl_k^{-\alpha}$ for some constant $C\ge 1$.
    Then, $\tilde e_k$ generated with the levels
    $\tilde l_k:=C^{1/\alpha}l_k$ satisfies
    $\tilde e_{k+1}\le c\tilde e_k+C \tilde l_k^{-\alpha}=c\tilde e_k+
    l_k^{-\alpha}$. %
    The cost quantity $\sum_{j=0}^{k-1}\tilde l_j$ only increases by
    the constant factor $C^{1/\alpha}$ compared to
    $\sum_{j=0}^{k-1}l_j$. Hence, the asymptotic cost behaviour stated
    in Theorem \ref{thm:SL} (single-level) and Theorem \ref{thm:ML}
    (multi-level) remains valid also for $C>1$.
  \end{remark}

  We next continue our discussion of Example \ref{ex:2}, for details
  see Appendix \ref{app:running_conv}.

\begin{example}[Continuation of Example \ref{ex:2}]\label{ex:3}
  The regularized objective $\Phi$ in Example \ref{ex:2} is
  $\lambda$-strongly convex and $L$-smooth with
  $L=\|\xi\|_{L^2(D)}^2\|\Gamma^{-1}\|+\lambda$. %
  Consider the multilevel gradient descent method
  $f_{j+1}=f_j-\eta g_{l_j}(f_j)$, where $g_l$ is the approximation to
  $\nabla\Phi$ from Example \ref{ex:2}. The level choice \eqref{eq:ML}
  then yields $\|f_{*}-f_{K(\eps)}\|_{L^2(D)}\lesssim \eps$, for the
  unique minimizer $f_*$ of $\Phi$. The cost quantity, which %
  corresponds to the aggregated computational cost of all required FEM
  approximations to compute $f_{K(\eps)}$, behaves like $\eps^{-1}$ as
  $\eps\to 0$.
\end{example}

Finally we point out that our analysis and notion of quasi-optimality
are based on the constraint $\tilde e_k\le\eps$ (rather than
$e_k\le\eps$), where $\tilde e_k$ is an upper bound of the actual
error $e_k$. In Appendix \ref{app:opt_rate} we give a concrete example
of biased gradient descent to show that the cost asymptotics in
\eqref{eq:costMLasymp} is in general sharp for the actual error $e_k$
as well.

\section{Particle based optimization: A multilevel ensemble Kalman
  inversion}\label{sec:MLEKI}
As our first application, we present a multilevel ensemble Kalman
inversion (EKI) to solve the problem presented in
Sec.~\ref{sec:regopt}. EKI is a derivative free particle based
optimization method, e.g., \cite{SchSt2017, BSWW19, KS2019}. We first
recall the method, and subsequently present a multilevel version.

\subsection{Ensemble Kalman inversion}
EKI refers to a specific dynamical system describing the evolution of
an ensemble of particles.  By the well-known subspace property
\cite{ILS2013}, these particles remain within the finite dimensional
affine subspace spanned by the ensemble at initialization. Therefore,
there is no loss of generality in assuming
$\mathcal X=\mathbb R^{n_x}$ finite dimensional throughout this
section.

We formulate the EKI as a method to minimize the objective
\begin{equation}\label{eq:EKIobjective}
  \Phi(x) = \frac12\|\Sigma^{-1/2}(H(x)-z)\|^2.
\end{equation}
Here $H:\mathbb R^{n_x}\to\mathbb R^{n_z}$ is the %
forward model %
and $\Sigma\in\mathbb R^{n_z\times n_z}$ is SPD. Letting
\begin{align}\label{eq:EKIdata}
  H=F,\qquad
  z=y\in\mathbb R^{n_y},\qquad
  \Sigma=\Gamma,
\end{align}
\eqref{eq:EKIobjective} corresponds to the %
objective in \eqref{eq:regularized_lossfunction} with $R=0$ (i.e.\
unregularized).  Fixing an SPD matrix $C_0\in R^{n_x\times n_x}$,
\begin{align}\label{eq:TEKIdata}
  H%
  = \begin{pmatrix}F%
    \\ {\mathrm{Id}}\end{pmatrix},\quad z = \begin{pmatrix}y\\ {\bf{0}}_{\mathbb R^{n_x}} \end{pmatrix},\quad \Sigma = \begin{pmatrix}\Gamma &0\\ 0 & \frac{1}{\lambda}C_0\end{pmatrix}, 
\end{align}
yields the regularized objective $\Phi(x)$ in
\eqref{eq:regularized_lossfunction} with regularizer
$R(x) = \frac{\lambda}{2}\|C_0^{-1/2}x\|^2$. We refer to EKI applied
to the unregularized and regularized objective as \emph{standard EKI}
and %
\emph{Tikhonov regularized EKI (TEKI)}, respectively. See
\cite{CST2019,WCST2021} for more details on %
TEKI.

We consider the continuous-time formulation of EKI
\cite{BSW2018,BSWW2021}: %
For a fixed ensemble size $M\in\N$, let $v_t^{(m)}\in\R^{n_x}$,
$m=1,\dots,M$, satisfy the coupled system of stochastic differential
equations (SDEs)
\begin{align}\label{eq:EKI_cont}
  \begin{split} 
    {\mathrm d}v_t^{(m)} &= C^{v,H}(v_t)\Sigma^{-1}(z-H(v_t^{(m)}))\,{\mathrm d}t + C^{v,H}(v_t)\Sigma^{-1/2}\,{\mathrm d}W_t^{(m)},\\
    &v_0^{(m)}\overset{\text{i.i.d.}}{\sim}Q_0 \qquad m=1,\dots,M.
  \end{split}
\end{align}
Here $W_t^{(m)}$ are independent $\mathbb R^{n_z}$-valued Brownian
motions, $Q_0$ %
is a fixed initial distribution with finite second moment on
$\R^{n_x}$, and $C^{v,H}\in\R^{n_x\times n_z}$ denotes a mixed sample
covariance, see Appendix~\ref{app:notation} for the precise formula. %
Under certain assumptions, it can be shown that %
\eqref{eq:EKI_cont} is well-posed, i.e.~existence of unique and strong
solutions %
can be guaranteed, and their average converges to the minimizer of
$\Phi$ (\cite{BSWW19}). %

To motivate %
this behaviour, suppress for the moment the diffusion term in
\eqref{eq:EKI_cont}, and consider a linear forward map %
{$H\in\mathbb R^{n_z\times n_x}$}. Then %
\begin{align*}
  \frac{{\mathrm d}v_t^{(m)}}{{\mathrm d}t} = -C(v_t)H^\top \Sigma^{-1}(Hv_t^{(m)}-z) = -C(v_t)\nabla_v\Phi(v_t^{(m)}),
\end{align*}
for $\Phi(x) = \frac12\|\Sigma^{-1/2}(Hx-z)\|^2$. Hence, in the linear
and deterministic setting the EKI %
is a preconditioned gradient flow w.r.t.~the data misfit $\cV$ (or
w.r.t.~the Tikhonov regularized data misfit).  We refer to
  \cite{Chada2019ConvergenceAO,W2022} for %
{more details} on the nonlinear setting.

\subsection{Multilevel ensemble Kalman inversion}
To formulate the multilevel EKI, assume given approximations $H_l$,
$l>0$, to $H$ in \eqref{eq:EKI_cont}.  {Fixing $\tau >0$, with
  $t_j:=j\cdot \tau $ we denote by $v_{t_{j+1}}^{(m),l_j}$ the
  solution of the coupled system of SDEs in integral form %
  \begin{equation}\label{eq:EKI_level_cont}
    v_{t_{j+1}}^{(m),l_j} \hspace{-2pt} =\hspace{-2pt} v_{t_j}^{(m),l_{j-1}}\hspace{-2pt}+\hspace{-2pt}\int_{t_j}^{t_{j+1}} C^{v,H_{l_j}}(v_t^{l_j})\Sigma^{-1}(z-H_{l_j}(v_t^{(m),l_j}))\,{\mathrm d}t + \int_{t_j}^{t_{j+1}}C^{v,H_{l_j}}(v_t^{l_j})\Sigma^{-1/2}\,{\mathrm d}W_t^{(m)},
  \end{equation}
  initialized via $v_{t_j}^{(m),l_j} = v_{t_{j}}^{(m),l_{j-1}}$.
  We then introduce a discrete-time process $(x_j)_{j}$ as the
  ensemble mean of the particles at time $t_j$:}
\begin{align}\label{eq:EKIx}
  x_{j} %
  := \frac1M\sum_{m=1}^M v_{t_j}^{(m),l_j}.
\end{align}

We discuss a discretization scheme for the SDE
\eqref{eq:EKI_level_cont} and summarize the multilevel EKI as an
algorithm in Appendix~\ref{app:algo_ML_EKI}.  {The goal in the
  following is to show that $x_j$ approximately minimizes the
  objective \eqref{eq:EKIobjective} for large $j$.}

\subsection{Error analysis for linear forward operators}
Assume $F$ to be linear, specifically $F\in\mathbb R^{n_y\times n_x}$
with $n_y<n_x$ and ${\rm{rank}}(FF^\top) = n_y$. We make the following
assumptions on the forward operator arising through numerical
approximation which is typically satisfied for ODE- or PDE-based
forward models. We verify this assumption for Example~\ref{ex:1} in
Appendix~\ref{app:running}, but emphasize that this is typically
satisfied and known for numerical approximations (e.g., using finite
elements or boundary elements) to forward maps described by PDE
solutions, see for example \cite{MR4269305,MR2743235}.
\begin{assumption}[Approximation of the forward
  operator]\label{ass:forwardmodel}
  There exist $b>0$ and $\alpha>0$ such that for each $l>0$ there
  exists $F_l\in\R^{n_y\times n_x}$ satisfying
  $\|F-F_l\|_{\mathbb R^{n_y\times n_x}}^2 \le b l^{-\alpha}$.
\end{assumption}

In the following we use the notation $H_l$ to denote $H$ as in
\eqref{eq:EKIdata} or \eqref{eq:TEKIdata} but with $F$ replaced by
$F_l$, i.e.~$H_l\in\mathbb R^{n_z\times n_x}$.  We consider EKI and
TEKI with \emph{covariance inflation}, which is a standard data
assimilation tool to stabilize the scheme, e.g.,
\cite{JLA07,JLA09,TMK2016}.  Specifically, for some fixed SPD matrix
$B\in\R^{n_x\times n_x}$, \eqref{eq:EKI_level_cont} is replaced by the
stabilized dynamics
\begin{equation}\label{eq:VI_EKI_level_cont}
  v_{t_{j+1} }^{(m),l_j} = \int_{t_j}^{t_{j+1} } (C(v_t^{l_j})+B)H_{l_j}^\top\Sigma^{-1}(z-H_{l_j}v_t^{(m),l_j})\,{\mathrm d}t + \int_{t_j}^{{t_{j+1}} }C(v_t^{l_j})H_{l_j}^\top\Sigma^{-1/2}\,{\mathrm d}W_t^{(m)}.
\end{equation}

Assume that an evaluation of $F_l$ has cost $O(l)$. Then,
approximating \eqref{eq:VI_EKI_level_cont} using for example an
Euler-Maruyama scheme with a fixed number of time steps $T$, requires
$O(M\cdot T)$ evaluations of $F_l$. With $M$ and $T$ interpreted as
constants, this amounts to cost $O(l)$. In this sense the cost
quantity $\sum_{j=0}^{K-1}l_j$ introduced in \eqref{eq:cost}, can be
interpreted as the computational cost of computing $x_K$ in
\eqref{eq:EKIx}. We give convergence result for EKI and TEKI in the
case of noise-free and noisy data respectively.

Let $x_j$ in \eqref{eq:EKIx} be the mean of the particle system driven
by \eqref{eq:VI_EKI_level_cont}, with $H$ (and $H_l$) as in
\eqref{eq:EKIdata}.

\begin{proposition}[Multilevel EKI]\label{prop:error_EKI}
  Let Assumption~\ref{ass:forwardmodel} be satisfied,
  $y=Fx^\dagger\in\R^{n_y}$ for some truth
  $x^\dagger\in\mathbb R^{n_x}$ and let %
  $Q_0$ have finite second moment on $R^{n_x}$. For $\tau >0$
  sufficiently small, there exists $c\in(0,1)$ depending on $F$ and
  $B$ such that for levels $l_{K,j}(\eps)$, $j=0,\dots,K(\eps)-1$,
  given by \eqref{eq:ML},
  \begin{equation*}
    e_{K(\eps)}:={\mathbb E[\Phi(x_{K(\eps)})]}-\Phi(x^\dagger)=
    \mathbb E[%
    \|\Gamma^{-1/2}(F x_{K(\eps)}-y)\|_{\R^{n_y}}^2]
    \le \varepsilon,
  \end{equation*}
  for all small enough $\eps>0$.  Furthermore,
  ${\rm cost}_{\rm ML}(\eps)=\sum_{j=0}^{K(\eps)-1}l_{K,j}(\eps)\simeq
  \eps^{-\frac{1}{\alpha}}$.
\end{proposition}
\begin{proof}
  For the proof see Appendix~\ref{app:proof_Prop_EKI}.
\end{proof}

Now let $x_j$ in \eqref{eq:EKIx} be the mean of the particle system
driven by \eqref{eq:VI_EKI_level_cont}, with $H$ (and $H_l$) as in
\eqref{eq:TEKIdata}.

\begin{proposition}[Multilevel TEKI]\label{prop:error_TEKI}
  Let Assumption~\ref{ass:forwardmodel} be satisfied,
  $x_\ast\in\R^{n_x}$ be the unique minimizer of $\Phi$ in
  \eqref{eq:EKIobjective} with $H$ given by \eqref{eq:TEKIdata}, and %
  let $Q_0$ have finite second moment on $R^{n_x}$.  For $\tau >0$
  sufficiently small there exists $c\in(0,1)$ depending on $F$ and $B$
  such that for levels $l_{K,j}(\eps)$, $j=0,\dots,K(\eps)-1$, given
  by \eqref{eq:ML},
  \begin{align*}
    e_{K(\eps)}:=\mathbb E\left[\frac12\|\Gamma^{-1/2}F(x_{K(\eps)}-x_\ast)\|_{\R^{n_y}}^2 + \frac{\lambda}{2}\|%
    C_0^{-1/2}(x_{K(\eps)}
    -x_\ast)\|_{\R^{n_x}}^2\right]\le\eps,
  \end{align*}
  for all small enough $\eps>0$. Furthermore, %
  ${\rm cost}_{\rm ML}(\eps)=\sum_{j=0}^{K(\eps)-1}l_{K,j}(\eps)\simeq
  \eps^{-\frac{1}{\alpha}}$.
\end{proposition}
\begin{proof}
  For the proof see Appendix~\ref{app:proof_Prop_TEKI}.
\end{proof}

In both cases there holds an analogous statement for the single-level
choice in \eqref{eq:SL} with (the worse) asymptotic cost behaviour
$\eps^{-1/\alpha}\log(\eps^{-1})$ as $\eps\to 0$.

\section{Optimization over probability measures: Multilevel Bayesian
  inference}\label{sec:MLsampling} %

As our second application, we consider an interacting particle system,
to solve the problem presented in Sec.~\ref{sec:baysinf}. To derive
the method, we apply the multilevel optimization of
Sec.~\ref{sec:ML_opti} to a gradient flow in the space of probability
measures.  Appendix \ref{app:algo_ML_Langevin} gives details on the
algorithm.

Throughout this section we %
adopt the assumptions of Sec.~\ref{sec:baysinf}, i.e.\ $\cX=\R^{n_x}$,
the forward model $F:\R^{n_x}\to\R^{n_y}$ may be nonlinear and the
prior $Q_0\sim \cN(0,\frac1\lambda C_0)$ is Gaussian.  We denote it's
density by $q_0$. The posterior density then equals
\begin{equation}\label{eq:posterior}
  \rho_\ast(x) = \frac{1}{Z}\exp(\cV_R(x)),\qquad
  \cV_R(x):=\frac12\|\Gamma^{-1/2}(F(x)-y)\|_{\R^{n_y}}^2+\frac{\lambda}2\|C_0^{-1/2}x\|_{\R^{n_x}}^2
\end{equation}
with $Z=\int_{\R^{n_x}}\exp(\cV_R(x))\,{\mathrm d}x$.  To explain the
idea of approximating $\rho_\ast$ by an ensemble of particles, we
first recall the Langevin dynamics.  Let $v_t\in\R^{n_x}$ initialized
as $v_0\sim q_0$ solve
\begin{equation}\label{eq:Langevin_SDE}
  {\mathrm d}v_t = -\nabla_x \cV_R(v_t)\,{\mathrm d}t + \sqrt{2}{\mathrm d}W_t,\quad v_0\sim q_0.
\end{equation}
The evolution of its distribution $\rho_t$ is described by the
Fokker-Planck equation, see e.g.~\cite{P14},
\begin{equation}\label{eq:FokkerPlanck}
  \partial_t \rho_t = \nabla\cdot (\rho_t\nabla \cV_R)+\Delta\rho_t,\quad \rho_0 = q_0.
\end{equation}
Under certain assumptions on $\cV_R$ the Markov process $v_t$ is
ergodic and $\rho_\ast$ is its unique invariant distribution.
Furthermore, it is possible to describe the rate of convergence in
terms of the gradient flow structure given by \eqref{eq:FokkerPlanck},
see for example \cite[Theorem~4.9]{P14}.

 \subsection{Interacting Langevin sampler}
 We consider an interacting Langevin particle %
 system introduced in \cite{GHWS2020}.  Let $(v_t^{(m)})_{m=1}^M$ %
 solve %
 \begin{equation}\label{eq:interactingLang}
   {\mathrm d} v_t^{(m)} = -C(v_t)\nabla\cV_R(v_t^{(m)})\,{\mathrm d}t + \sqrt{2C(v_t)}\,{\mathrm d}W_t^{(m)},
 \end{equation}
 with $v_0^{(m)}\sim q_0$ i.i.d., where $(W_t^{(m)})_{m=1}^M$ denote
 independent Brownian motions on $\R^{n_x}$ and with $C(v)$ the
 empirical covariance, see Appendix \ref{app:notation}. %
 Observe that the mean field limit satisfies the following
 preconditioned version of
 \eqref{eq:Langevin_SDE}-\eqref{eq:FokkerPlanck} (e.g.\
 \cite{DL2021_b}):
 \begin{align*}{\mathrm d} v_t = -C(\rho)\nabla\cV_R(v_t)\,{\mathrm
     d}t + \sqrt{2C(\rho)}\,{\mathrm d}W_t,\end{align*} where
 \begin{equation}\label{eq:mC}
   m(\rho)=\int_{\R^{n_x}} v\rho(v)\,{\mathrm d}v,\qquad C(\rho) = \int_{\R^{n_x}}\left((v-m(\rho))\otimes (v-m(\rho))\right)\rho(v)\,{\mathrm d}v
 \end{equation}
 and %
 \begin{equation}\label{eq:prec_FokkerPlanck}
   \partial_t \rho_t = \nabla\cdot (\rho_tC(\rho_t)\nabla \cV_R)+\Tr(C(\rho_t){\mathrm D}^2\rho_t),\quad \rho_0 = q_0.
 \end{equation}

\subsection{Mean-field and discrete multilevel interacting Langevin sampler}
Fix $\tau >0$ and set $t_j:=j\cdot\tau$.  In the following let
$\cV_R^l$ be given by \eqref{eq:posterior} with the forward model
$F:\R^{n_x}\to\R^{n_y}$ replaced by an approximation $F_l$.  We
consider the discrete time process $(\rho_j)_{j=1,\dots,K}$
iteratively defined by $\rho_{j+1} = \rho_{{t_{j+1}} }^{l_j},$
where $\rho_0^{l_1}=q_0$ and $\rho_{t}^{l_j}$ solves
\eqref{eq:prec_FokkerPlanck} on a time interval of length $\tau$
initialized with the previous distribution, i.e.\
\begin{equation}\label{eq:FokkerPlanck_l}
  \partial_t \rho_t^{l_j} = \nabla\cdot (\rho_t^{l_j} C(\rho_t^{l_j})\nabla \cV_R^{l_j})+\Tr(C(\rho_t^{l_j}){\mathrm D}^2\rho_t^{l_j}),\quad \rho_{t_j}^{l_j} = \rho_{t_{j}}^{l_{j-1}}.
\end{equation}

While the subsequently discussed analysis will be based on the mean
field limit \eqref{eq:FokkerPlanck_l}, we present a discretized
version by introducing a particle based approximation based on
\eqref{eq:interactingLang}.  To this end let $(\hat \rho_j)_j$
iteratively be defined by %
\begin{align*} \hat \rho_{j+1} = \frac{1}{M}\sum_{m=1}^M
  \delta_{v_{t_{j+1}}^{(m),l_j}}, %
\end{align*}
where $\delta_v$ denotes the dirac-measure at $v$ and
$v_{{t_{j+1}} }^{(m),l_j}$ initialized by
$v_{t_j}^{(m),l_j} = v_{t_{j} }^{(m),l_{j-1}}$ solves
\begin{align}\label{eq:interacting_Langevin_integral}
  v_{{t_{j+1}}}^{(m),l_j} = v_{t_j}^{(m),l_{j}} -\int_{t_j}^{{t_{j+1}}} C(v_t^{l_j})\nabla \cV_R(v_t^{(m),l_{j}})\,{\mathrm d}t + \int_{t_j}^{t_{j+1}} \sqrt{2C(v_t^{l_j})}\,{\mathrm d}W_t^{(m)},\quad {m=1,\dots,M}.
\end{align}

\subsection{Mean-field error analysis for the Kullback--Leibler divergence}
Under convexity assumptions on $\cV_R^{l}$ and assuming that
$C(\rho_t^l)$ does not degenerate, one can prove exponential
convergence to equilibrium for the mean-field limit of the interacting
Langevin dynamics for any fixed level $l$
\cite[Proposition~2]{GHWS2020}. These assumptions can for example be
verified for linear forward models and Gaussian priors
\cite[Corollary~5]{GHWS2020}. We make the following additional
assumption:
\begin{assumption}[Approximation of the forward
  operator]\label{ass:KL}
  There exist $b>0$ and $\alpha>0$ such that for each $j\in\N$ there
  exists $F_{l_j}:\R^{n_x}\to \R^{n_y}$ satisfying
  \begin{align*} \int_{\mathbb R^{n_x}} \|F(x)-F_{l_j}(x)\|^2
    \rho_j(x)\,{\mathrm d}x \le bl_j^{-\alpha}.\end{align*} %
\end{assumption}
For linear forward models and Gaussian priors, the pdf $\rho_j$
remains Gaussian and Assumption~\ref{ass:KL} can be inferred from
Assumption~\ref{ass:forwardmodel}. Extending the convergence result
of~\cite{GHWS2020} to error decay \eqref{eq:disc} allows us to apply
our results from Sec.~\ref{sec:ML_opti}. This leads to the following
Proposition:

\begin{proposition}[Multilevel interacting Langevin sampler]\label{prop:error_EKS}
  Suppose $\KL{q_0}{\rho_*}<\infty$ and let Assumption~\ref{ass:KL} be
  satisfied.  Suppose there exist $\sigma_1$, $\sigma_2>0$ such that
  $\nabla^2 \cV_R^l>\sigma_1{\mathrm{Id}}$ and
  $C(\rho_j)>\sigma_2{\mathrm{Id}}$ (cp.~\eqref{eq:mC}) for all
  $j\ge 0$. For $\tau >0$ sufficiently small there exists $c\in(0,1)$
  depending on $\sigma_1,\sigma_2$, such that $l_{K,j}(\eps)$,
  $j=0,\dots,K(\eps)-1$ given by \eqref{eq:ML} yields
  \begin{equation*}
    e_{K(\eps)}:=
    \KL{\rho_{K(\eps)}}{\rho_\ast}
    \le \varepsilon,
  \end{equation*}
  for all small enough $\eps>0$ where $\rho_j$ solves
  \eqref{eq:FokkerPlanck_l}.  Moreover,
  ${\rm cost}_{\rm
    ML}(\eps)\simeq\sum_{j=0}^{K(\eps)-1}l_{K,j}(\eps)\simeq
  \eps^{-\frac{1}{\alpha}}$.
\end{proposition}
\begin{proof}
  For the proof see Appendix~\ref{app:proof_Prop_Langevin}.
\end{proof}

An analogous results holds for the single-level choice in
\eqref{eq:SL} with cost behavior
$\log(\varepsilon^{-1})\varepsilon^{-\frac1\alpha}$.

\section{Numerical experiments}\label{sec:numerics}
We consider an adaptation of %
our running example in a one-dimensional setting with $D=(0,1)$.
Instead of working on $L^2(D)$ directly, we parametrize $L^2(D)$ via
\begin{align*}
  f(x) = f(x,\cdot) = \sum_{i\in\N} x_i\frac{\sqrt{2}}{\pi}\sin(i\pi \cdot)\in L^2(D),
\end{align*}
for $x\in\R^{n_x}$. %
We used the observation operator
$\cO:H^2(D)\to \R^{n_y}:u_f\mapsto (u_f(s_i))_{i=1}^{n_y}$ with the
equispaced points $s_i = \frac{i}{n_y+1}$, $i=1,\dots,n_y=15$. The
forward model is $F(x)=\cO(u_{f(x)})\in\R^{n_y}$ for $x\in\R^{n_x}$,
where $u_f$ denotes the solution of \eqref{eq:PDE}. We used piecewise
linear FEM on a uniform mesh to approximate $F$ by $F_l$.  This setup
corresponds to $\alpha=1$ in Sec.~\ref{sec:ML_opti}. As a prior on the
parameter space $\R^{n_x}$ we chose $Q_0=\cN(0,C_0)$ with
$C_0 = {\mathrm{diag}}(i^{-2},i=1,\dots,n_x)$.

We ran multilevel TEKI (Algorithm \ref{alg:MLEKI}) and the multilevel
interacting Langevin sampler (Algorithm \ref{alg:MLLangevin}) on this
problem and plotted the error convergence in
Fig.~\ref{fig:complexity_comparison}.  For TEKI, the plotted error
quantity is
$\mathbb
E[\frac12\|F_{\mathrm{ref}}(x_{K}(\varepsilon)-x_\ast)\|_\Gamma^2
+ \frac{\lambda}{2}\|x_K(\varepsilon)-x_\ast\|_{C_0}^2]$ with
$F_{\mathrm{ref}}=F_{2^{14}}$. For the interacting Langevin
sampler we consider the convergence of the posterior mean, and the
error quantity is
$\mathbb E[\frac12\|f(\cdot,
x_{K(\varepsilon)})-f(\cdot,x_\ast)\|_{L^2(D)}^2]$. The observed
convergence rates roughly coincide with the ones proven in
Sec.~\ref{sec:MLEKI} and Sec.~\ref{sec:MLsampling}. In particular, the
multilevel algorithms are superior to their single-level counterparts.
Even though the reduction in cost is only by the factor
$\log(\eps^{-1})$, as the plot shows this can amount to significant
gains in practice.  More details on all chosen parameters and the
setup can be found in Appendix \ref{app:numerics}.
\begin{figure}[!htb]
  \begin{center}
    \includegraphics[width=.42\textwidth]{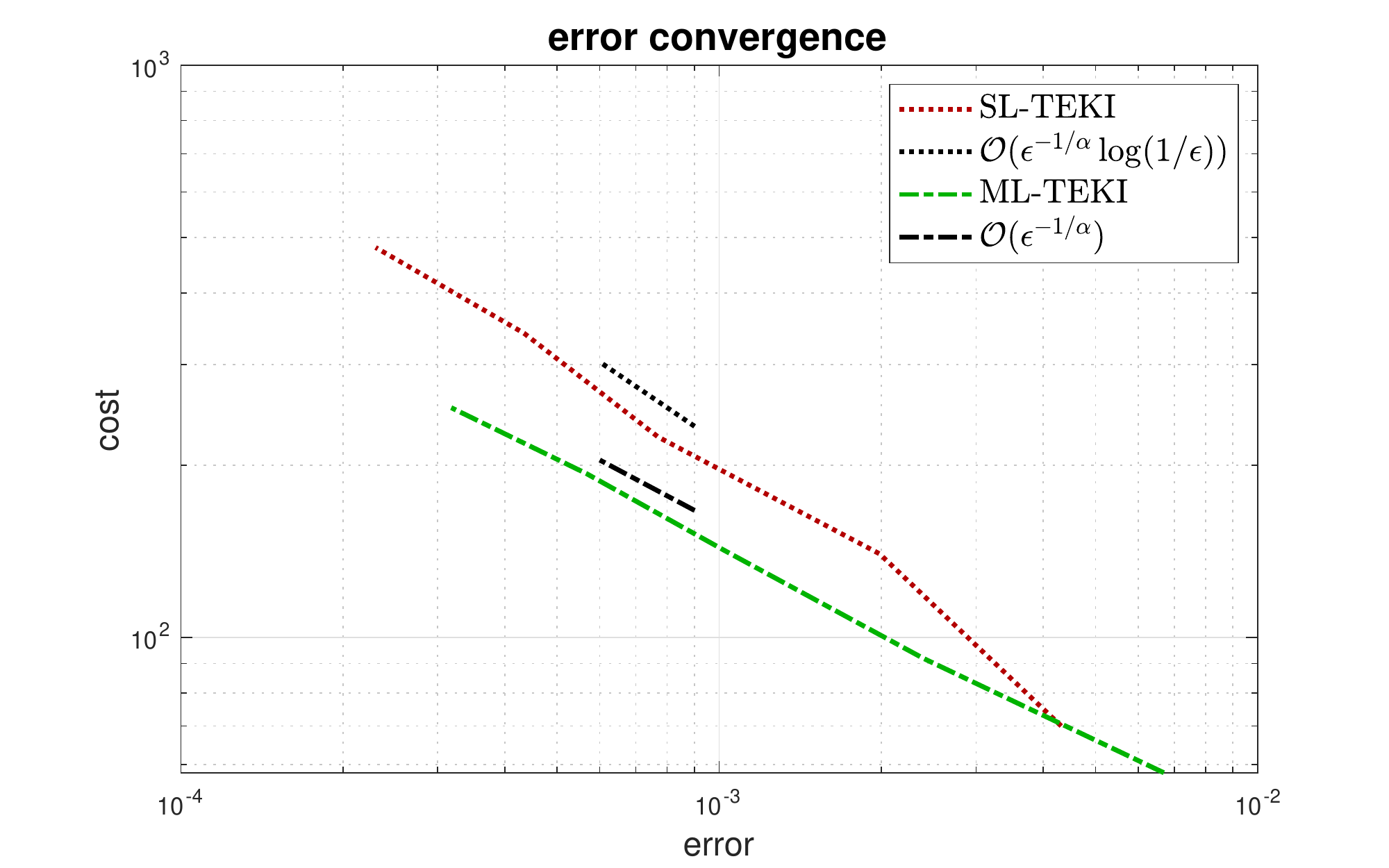}
    \includegraphics[width=.42\textwidth]{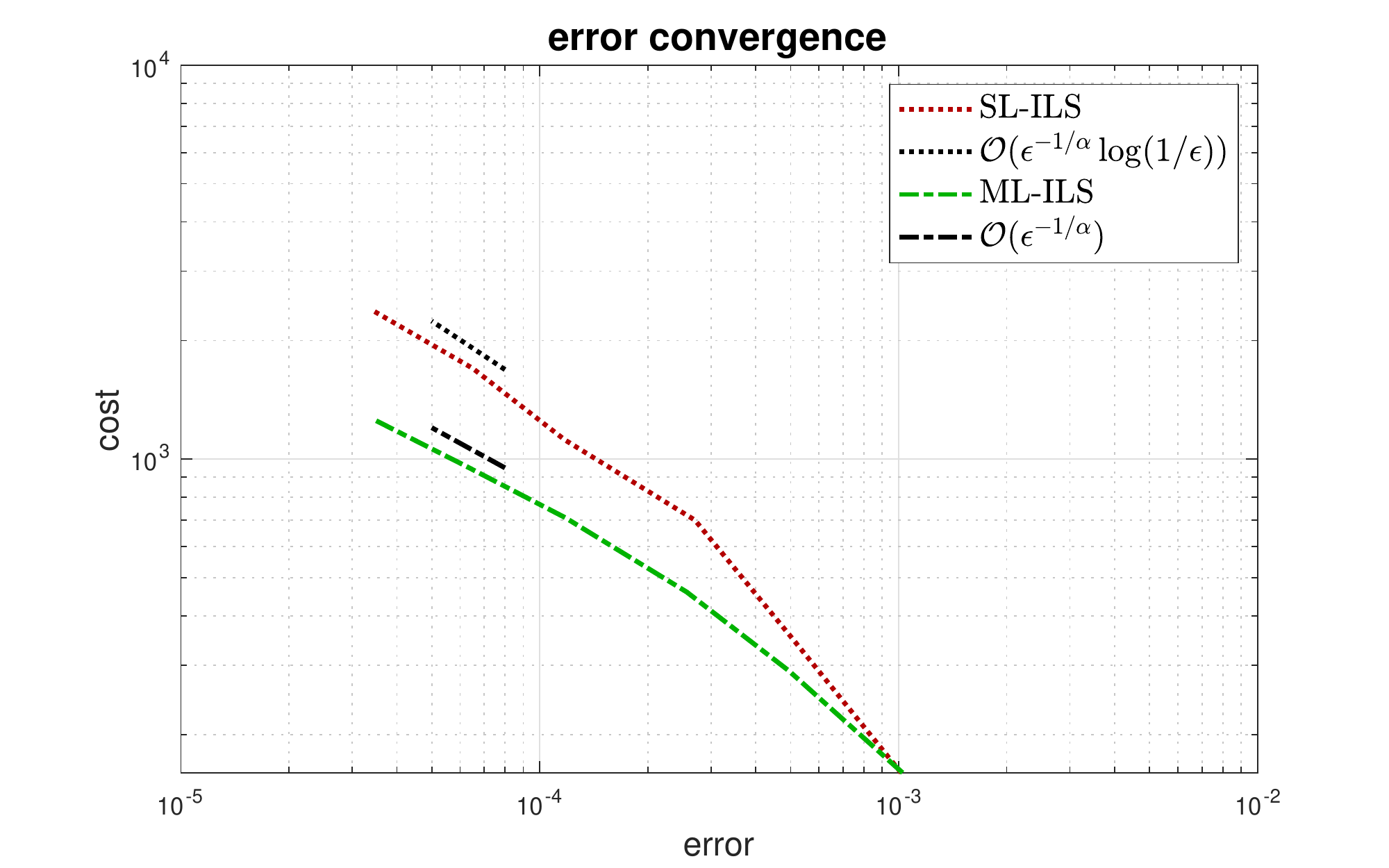}
    \caption{Computational costs vs.~estimated error for TEKI (left)
      and the interacting Langevin sampler (ILS) (right). We ran all
      Algorithms $100$ times to estimate the expected error.  }
    \label{fig:complexity_comparison}
  \end{center}
\end{figure}

\section{Conclusion}
We have presented an abstract multilevel optimization framework, and
provided quasioptimal level choices. We showed improved convergence
compared to single-level methods, and demonstrated the wide
applicability by introducing a novel multilevel ensemble Kalman
inversion, as well as a new multilevel Langevin sampler. While our
main focus was on inverse problems, we additionally discussed
different versions of multilevel gradient descent, which in principle
are applicable to any kind of optimization problem where the
evaluation of the (derivatives of the) objective function is expensive
or impossible, and demands approximation by numerical methods.
Further directions or research could include developing optimization
procedures for sampling with respect to other metrics (than KL), such
as the Wasserstein distance. The proof could proceed along similar
lines as Proposition~\ref{prop:error_EKS}.  Moreover, developing a
framework for adaptive step sizes and levels would be an interesting
extension.

\bibliographystyle{abbrv}
\bibliography{references.bib}

\newpage

\appendix

\section{Proofs of the main theorems }

\subsection{Proof of Theorem \ref{thm:SL}}\label{app:proof_Prop_SL}
\begin{proof} Let $\eps\in (0,e_0)$. With $\tilde e_K$ in
  \eqref{eq:ek}, denote
  \begin{equation}\label{eq:eK12}
    \tilde e_K((l_j)_j) =
    c^{K} e_0 + \sum_{j=0}^{K-1} c^{K-1-j} l_j^{-\alpha}
    =:\tilde e_{K,1} +\tilde e_{K,2}((l_j)_j).
  \end{equation}
  By definition, $K(\eps)\in\N$ in \eqref{eq:SL} is minimal such that
  $\tilde e_{K(\eps),1}\le\frac{\eps}{2}$. Moreover, by definition
  ${\bar l_K}(\eps)>0$ in \eqref{eq:SL} is minimal such that
  $\tilde e_{K,2}(({\bar l_K}(\eps))_j)\le\eps$.

  Now let $\hat K\in\N$ and $\hat l>0$ be any numbers such that
  $\tilde e_{\hat K}((\hat l)_j)\le\eps$.  Then
  \begin{equation*}
    \tilde e_{\hat K,1}\le
    \tilde e_{\hat K,1}+
    \tilde e_{\hat K,2}((\hat l)_j)
    =
    \tilde e_{\hat K}((\hat l)_j)\le\eps,
  \end{equation*}
  and by the optimality property of $K(\eps)$ we obtain
  $\hat K\ge K(2\eps)$. Similarly, since
  \begin{equation}
    \tilde e_{\hat K,2}((\hat l)_j)
    \le \eps,
  \end{equation}
  the optimality property of ${\bar l_K}(\eps)$ implies
  $\hat l\ge {\bar l_K}(2\eps)$.  Taking the infimum over all such
  $\hat l$ and $\hat K$
  \begin{equation}\label{eq:K2eps}
    K(2\eps){\bar l_K}(2\eps) \le
    \inf\{\hat K\hat l\,:\,\tilde e_K(\hat l)\le\eps\}
    \le K(\eps){\bar l_K}(\eps),
  \end{equation}
  where the second inequality holds due to
  $\tilde e_{K(\eps)}({\bar l_K}(\eps))\le \eps$.

  By definition of $K(\eps)$ and ${\bar l_K}(\eps)$ in \eqref{eq:SL}
  \begin{equation*}
    {\rm cost}_{\rm SL}(\eps) = K(\eps) l_{K}(\eps) = \log(c)^{-1}\log\left(\frac{\varepsilon}{2e_0}\right)\left(\sum_{i=0}^{K-1}c^{\frac{K-1-i}{1+\alpha}}\right)^{\frac1\alpha} \left(\frac{\varepsilon}{2}\right)^{-\frac1\alpha}.
  \end{equation*}
  Since with $\delta:=c^{\frac1{1+\alpha}}\in (0,1)$ holds
  $1\le \sum_{i=0}^{K-1}c^{\frac{K-1-i}{1+\alpha}} \le (1-c)^{-1}$ we
  have
  ${\rm cost}_{\rm SL}(\eps)\simeq
  \eps^{-\frac{1}{\alpha}}\log(\eps^{-1})$ as $\eps\to 0$. Together
  with \eqref{eq:K2eps} we find
  \begin{equation*}
    {\rm cost}_{\rm SL}(\eps) \simeq \eps^{-\frac{1}{\alpha}}\log(\eps^{-1})
    \simeq
    \inf\{\hat K\hat l\,:\,\tilde e_{\hat K}(\hat l)\le\eps\}\qquad\text{as}\quad\eps\to 0.
  \end{equation*}
  This shows \eqref{eq:costSL} and quasi-optimality in the sense of
  Def.~\ref{def:qosl}.
\end{proof}

\subsection{Proof of Lemma \ref{lemma:ML}}\label{app:proof_Lemma_ML}

\begin{proof}%
  Define $a_j:= c^{K-1-j}$ for all $j=0,\dots,K-1$.  We wish to
  minimize $\sum_{j=0}^{K-1}l_j$ for $l_j>0$ and under the constraint
  $\sum_{j=0}^{K-1} a_j l_j^{-\alpha}-\varepsilon$. To this end we use
  a Lagrange multiplier and consider
  \begin{equation*}
    \min_{l_1,\dots,l_K,\lambda}\ \sum_{j=0}^{K-1} l_j + \lambda \Bigg(\sum_{j=0}^{K-1} a_j l_j^{-\alpha}-\varepsilon\Bigg).
  \end{equation*}
  Taking the derivatives w.r.t.~$l_j$ and $\lambda$ leads to the
  following first order optimality conditions
  \begin{equation*}
    1-\lambda\alpha a_j l_j^{-(1+\alpha)} = 0,\qquad
    \sum_{j=0}^{K-1} a_j l_j^{-\alpha}-\varepsilon = 0.
  \end{equation*}
  The first condition gives
  $l_j = C_{K,\eps} a_j^{\frac{1}{1+\alpha}}$ for some constant
  $C_{K,\eps}$. Plugging $l_j$ into the second condition we find
  \begin{equation*}
    C_{K,\eps}^{-\alpha}\sum_{j=0}^{K-1} a_j a_j^{-\frac{\alpha}{1+\alpha}} = \varepsilon.
  \end{equation*}
  Hence
  $C_{K,\eps}=\varepsilon^{-\frac{1}{\alpha}}(\sum_{j=0}^{K-1}
  a_j^{\frac{1}{1+\alpha}})^{\frac{1}{\alpha}}$. Finally,
  \begin{equation*}
    \sum_{j=0}^{K-1}a_j^{\frac{1}{1+\alpha}}=\sum_{j=0}^{K-1}c^{\frac{j}{1+\alpha}}=
    \frac{1-c^{\frac{K}{1+\alpha}}}{1+c^{\frac{1}{1+\alpha}}}.
  \end{equation*}
  This shows \eqref{eq:optimal_level}.
\end{proof}

\begin{remark}\label{rmk:C}
  The constant $C_{K,\eps}$ in \eqref{eq:optimal_level} increases for
  decreasing tolerance $\varepsilon>0$. Moreover, $C_{K,\eps}$ is
  bounded from below and above uniformly for all $K\in\N$ as for
  $c\in(0,1)$ holds
  \begin{equation*}
    \sum_{j=0}^{K-1} c^{\frac{K-1-j}{1+\alpha}} %
    = \frac{1-(c^{\frac{1}{1+\alpha}})^{K}}{1-c^{\frac{1}{1+\alpha}}} \in \left(1,\frac{1}{1-c^{\frac1{1+\alpha}}}\right).
  \end{equation*}
\end{remark}

\subsection{Proof of Theorem \ref{thm:ML}}\label{app:proof_Prop_ML}
\begin{proof}%
  With $l_{K,j}(\eps)$ and $K(\eps)$ as in \eqref{eq:ML}, the
  calculation in \eqref{eq:costML} and Rmk.~\ref{rmk:C} show
  \begin{equation}\label{eq:propMLasymp}
    {\rm cost}_{\rm ML}(\eps)=\sum_{j=0}^{K(\eps)-1}l_{K,j}
    =\left(\frac{\eps}{2}\right)^{-\frac{1}{\alpha}}\left(\frac{1-c^{K(\eps)/(1+\alpha)}}{1-c^{1/(1+\alpha)}} \right)^{\frac{1+\alpha}{\alpha}}\simeq \eps^{-\frac{1}{\alpha}}\qquad\text{as}\quad\eps\to 0,
  \end{equation}
  as claimed.

  The proof of quasi-optimality in the sense of Def.~\ref{def:qoml}
  follows the same argument as in the proof of Theorem \ref{thm:SL}:
  Let $\eps\in (0,e_0)$. Split $\tilde e_K((l_j)_j)$ in \eqref{eq:ek},
  in the terms $\tilde e_{K,1}$ and $\tilde e_{K,2}((l_j)_j)$ as in
  \eqref{eq:eK12}.

  Now let $\hat K\in\N$ and $(\hat l_{j})_{j=0}^{\hat K-1}>0$ be
  arbitrary such that $\tilde e((\hat l_{j})_{j})\le\eps$.  Moreover,
  define
  \begin{equation*}
    \tilde l_{K,j}:= C_{\hat K,\eps}\cdot c^{\frac{\hat K-1-j}{1+\alpha}},\quad
    C_{\hat K,\eps} = \eps^{-\frac1\alpha} \left(\sum_{i=0}^{K-1} c^{\frac{\hat K-1-i}{1+\alpha}}\right)^{\frac1\alpha}
  \end{equation*}
  which by Lemma \ref{lemma:ML} minimizes
  $\sum_{j=0}^{\hat K-1}\tilde l_{K,j}$ under the constraint
  $\tilde e_{\hat K,2}((\tilde l_{K,j})_j)\le\eps$. Since also
  $\tilde e_{\hat K,2}((\hat l_{j})_j) \le\tilde e_{\hat K}((\hat
  l_{K,j})_j) \le\eps$, this implies
  \begin{equation}\label{eq:qoml1}
    \sum_{j=0}^{\hat K-1}\hat l_{j}\ge
    \sum_{j=0}^{\hat K-1}\tilde l_{K,j} = \eps^{-\frac{1}{\alpha}}\left(\frac{1-c^{\hat K/(1+\alpha)}}{1-c^{1/(1+\alpha)}} \right)^{\frac{1+\alpha}{\alpha}}
  \end{equation}
  where the equality holds by the calculation in \eqref{eq:costML}.

  Next observe that by definition, $K(2\eps)\in\N$ in \eqref{eq:ML} is
  minimal such that $\tilde e_{K(\eps),1}\le\eps$.  Since also
  $\tilde e_{\hat K,1}\le\tilde e_{\hat K}((\hat l_{j})_j)\le\eps$,
  this gives $K(2\eps)\le\hat K$. Therefore
  \begin{equation}\label{eq:qoml2}
    \eps^{-\frac{1}{\alpha}}\left(\frac{1-c^{\hat K/(1+\alpha)}}{1-c^{1/(1+\alpha)}} \right)^{\frac{1+\alpha}{\alpha}}
    \ge
    \eps^{-\frac{1}{\alpha}}\left(\frac{1-c^{K(2\eps)/(1+\alpha)}}{1-c^{1/(1+\alpha)}} \right)^{\frac{1+\alpha}{\alpha}}= \sum_{j=0}^{K(2\eps)} l_{K,j}(2\eps),
  \end{equation}
  where the last inequality holds by the calculation in
  \eqref{eq:costML}. In all, \eqref{eq:qoml1} and \eqref{eq:qoml2}
  yield
  \begin{equation*}
    \sum_{j=0}^{K(2\eps)-1}l_{K,j}(2\eps)\le \sum_{j=0}^{\hat K} \hat l_{j}.
  \end{equation*}
  Since $\hat K$ and $\hat l_{j}$ were arbitrary such that
  $\tilde e((\hat l_{j})_{j})\le\eps$,
  \begin{equation}\label{eq:allequalml}
    \sum_{j=0}^{K(2\eps)-1}l_{K,j}(2\eps)\le
    \inf\bigg\{\sum_{j=0}^{\hat K-1}\hat l_{j}\,:\,\tilde e_{\hat K}((\hat l_j)_j)\le\eps,~\hat K\in\N,~\hat l_j>0~\forall j\bigg\}
    \le \sum_{j=0}^{K(\eps)-1}l_{K,j}(\eps),
  \end{equation}
  where the second inequality holds due to
  $\tilde e_{K(\eps)}((l_{K,j}(\eps)))\le\eps$. By
  \eqref{eq:propMLasymp} all terms in \eqref{eq:allequalml} behave
  like $\eps^{-1/\alpha}$ as $\eps\to 0$, and this shows
  quasi-optimality in the sense of Def.~\ref{def:qoml}.
\end{proof}

\section{Gradient descent and accelerated gradient descent %
}\label{app:gdagd}

In this appendix we discuss the implications of our results for
gradient descent, accelerated gradient descent and the stochastic
versions of those algorithms. Throughout this section we assume that
$\cX$ is a Hilbert space and $\Phi:\cX\to\R$ is Fr\'echet
differentiable, $\mu$-strongly convex, i.e.\
\begin{equation}\label{eq:musc}
  \Phi(y)\ge\Phi(x)+\langle \nabla\Phi(x),y-x\rangle+\frac{\mu}{2}\|y-x\|^2\qquad\forall x,y\in\cX,
\end{equation}
and satisfies $L$-smoothness so that
\begin{equation}\label{eq:Lsmooth}
  \Phi(y)\le\Phi(x)+\langle\nabla\Phi(x),y-x\rangle+\frac{L}{2}\|y-x\|^2
  \qquad\forall x,y\in\cX.
\end{equation}
Moreover, we denote by $x_*\in\cX$ the unique minimizer of $\Phi$.

For the deterministic algorithms we assume the existence of
approximate gradients $g_l:\cX\to\cX$ such that
\begin{equation}\label{eq:ass2}
  \|\nabla\Phi(x)-g_l(x)\|\le \frac{l^{-\alpha}}{\eta}\qquad\forall x\in\cX,
\end{equation}
where $\eta>0$ will denote the step-size in the following.  For the
stochastic variants we will work under the assumption that for every
$x\in\cX$ there exists a random variable $G_l(x)\in\cX$ such that
\begin{equation}\label{eq:ass3}
  \mathbb{E}\|\nabla\Phi(x)-G_l(x)\|\le \frac{l^{-\alpha}}{\eta}\qquad\forall x\in\cX.
\end{equation}

\subsection{Gradient Descent (GD)}
The gradient descent update with step-size $\eta$ and approximate
gradient at level $l_k$ reads
\begin{equation}\label{eq:GD}
  x_{k+1} = x_k - \eta g_{l_k}(x_k).
\end{equation}

\begin{proposition}\label{prop:GD}
  Let $0 \leq \eta \leq 1/L$ and let $g_l$ satisfy \eqref{eq:ass2}.
  Then $x_k$ generated by \eqref{eq:GD} satisfies the
  bound~\eqref{eq:disc} with $e_k = \|x_{k+1} - x_*\|$ and
  $c =\sqrt{1-\eta \mu}$.
\end{proposition}
\begin{proof}
  Observe that
  \begin{align*}
    \|x_{k+1}-x_\ast\|  &= \|x_k-x_\ast - \eta\nabla_x\Phi(x_k) + \eta \nabla_x\Phi(x_k)-\eta g_{l_k}(x_k) \|\\
                        &\le\|x_{k}-x_\ast - \eta\nabla_x\Phi(x_k)\| + \eta\|\nabla_x\Phi(x_k)-g_{l_k}(x_k)\|.
  \end{align*}
  Using $\mu$-strong convexity \eqref{eq:musc} and $L$-smoothness
  \eqref{eq:Lsmooth} we obtain the following upper bound,
  \begin{align*}
    \|x_k-x_\ast - \eta\nabla_x\Phi(x_k)\|^2 &= \|x_k-x_\ast\|^2 - 2\eta\langle x_k-x_\ast,\nabla_x \Phi(x_k)\rangle  + \eta^2\|\nabla_x\Phi(x_k)\|^2\\ 
                                             &\le \|x_k-x_\ast\|^2 - \eta \mu \|x_k-x_\ast\|^2- 2\eta (\Phi(x_k)-\Phi(x_\ast)) + \eta^2 \|\nabla_x\Phi(x_k)\|^2\\
                                             &\le (1 - \eta \mu) \|x_k-x_\ast\|^2 - \eta(1/L-\eta)\|\nabla_x\Phi(x_k)\|^2 \\
                                             &\le (1 - \eta \mu) \|x_k-x_\ast\|^2.
  \end{align*}
  Combining these inequalities with assumption~\eqref{eq:ass2} leads
  to desired recursion~\eqref{eq:disc}
  \begin{equation*}
    e_{k+1} = \|x_{k+1}-x_\ast\|
    \le \sqrt{1-\eta \mu}\|x_k-x_\ast\|+  l_k^{-\alpha}  = c e_{k}+l_k^{-\alpha}.
  \end{equation*}
\end{proof}

It now follows from Thm.~\ref{thm:ML} that:

\begin{corollary}[MLGD]\label{cor:MLGD}
  Consider the setting of Proposition~\ref{prop:GD}. Then with the
  levels $l_{K,j}(\eps)$ as in \eqref{eq:ML}, $x_K$ generated by
  \eqref{eq:GD} satisfies $e_K:=\|x_K-x_*\|\le\eps$, and it holds
  $\sum_{j=0}^{K-1}l_{K,j}(\eps)=O(\eps^{-\alpha})$ as $\eps\to 0$.
\end{corollary}

\subsection{Stochastic Gradient Descent (SGD)}
\label{sec:stochgrad_details}
We now consider the stochastic setting, i.e.\ we assume given random
variables $G_l(x)$ as in \eqref{eq:ass3}.  The stochastic gradient
descent update with step-size $\eta$ and approximate stochastic
gradient at level $l_k$ then reads
\begin{equation}\label{eq:SGD}
  x_{k+1} = x_k - \eta G_{l_k}(x_k).
\end{equation}
\begin{proposition}\label{prop:SGD}
  Let $0 \leq \eta \leq 1/L$ and let $G_l$ satisfy \eqref{eq:ass3}.
  Then $x_k$ generated by \eqref{eq:SGD} satisfies the bound
  \eqref{eq:disc} with $e_k = \mathbb{E}[\|x_{k+1} - x_*\|]$ and
  $c =\sqrt{1-\eta \mu}$.
\end{proposition}
\begin{proof}
  The proof proceeds in the exact same manner as for gradient descent
  (taking expectations and replacing $g_{l_k}(x_k)$ with
  $G_{l_k}(x_k)$). To be more precise, taking the expectation
  w.r.t. the filtration up to iteration $k$, i.e.
  $\mathbb E_k[\cdot]:=\mathbb E[\cdot\mid
  \sigma(G_{l_j}(x_j),j=1,\dots,k-1)]$, we have
  \begin{equation*}
    \mathbb E_k[\|x_{k+1}-x_\ast\|] \le \sqrt{1-\eta \mu}\|x_k-x_\ast\| + \eta \mathbb E_k[\|\nabla_x\Phi(x_k)-G_{l_k}(x_k))\|],
  \end{equation*}
  where we again applied the inequality
  $\|x_k-x_\ast - \eta\nabla_x\Phi(x_k)\|^2\le (1-\eta
  \mu)\|x_k-x_\ast\|^2$ under $\mu$-strong convexity and
  $L$-smoothness.  Taking the expectation (this time without
  conditioning), we obtain
  \begin{equation*}
    e_{k+1} = \mathbb E[\|x_{k+1}-x_\ast\|] \le \sqrt{1-\eta \mu}\mathbb E[\|x_k-x_\ast\|] + \eta \mathbb E[\|\nabla_x\Phi(x_k)-G_{l_k}(x_k))\|]\le ce_k + l_k^{-\alpha}.
  \end{equation*}  
\end{proof}

  \begin{corollary}[MLSGD]\label{cor:MLSGD}
    Consider the setting of Proposition~\ref{prop:SGD}. Then with the
    levels $l_{K,j}(\eps)$ as in \eqref{eq:ML}, $x_K$ generated by
    \eqref{eq:SGD} satisfies $e_K:=\mathbb{E}[\|x_K-x_*\|]\le\eps$,
    and it holds $\sum_{j=0}^{K-1}l_{K,j}(\eps)=O(\eps^{-\alpha})$ as
    $\eps\to 0$.
  \end{corollary}

  We next discuss a standard example of stochastic gradient descent,
  namely with $G_l$ being a Monte Carlo estimator. However, we
  emphasize that other approximation schemes are applicable as well in
  this setting.
  \begin{example}[SGD with dynamic sampling]\label{ex:dynamicsgd}
    We consider a stochastic optimization problem in the form of
    \eqref{eq:general_opti} by
    \begin{equation}\label{eq:stoch_opti}
      \min_{x\in\mathcal X}\ \Phi(x),\quad \Phi(x):= \mathbb E_\xi[\varphi(x,\xi)],
    \end{equation}
    where $\xi$ is a random variable on a underlying probability space
    $(\Omega,\mathcal F,\mathbb P)$ with state space $(E,\mathcal E)$
    and $\varphi:\mathcal X\times E\to \mathbb R$ is the stochastic
    objective function. %
    The expected value $\mathbb E_\xi[\varphi(x,\xi)]$ and its
    gradient are often not available analytically.  Assuming %
    access to i.i.d.~samples $\{\xi_{k}^n\}_{n=1}^{l_k}$ of $\xi$, we
    can define a stochastic gradient approximation by the Monte Carlo
    estimator
    $G_{l_k}(x_k):=\frac{1}{l_k} \sum_{n=1}^{l_k}
    \nabla_x\varphi(x_k,\xi_{k}^n)$. %
    The level $l_k$ %
    describes the batch size %
    in iteration $k$, which correspond to the number of required
    evaluations of $\varphi$. In this sense, $l_k$ describes the
    computational cost required to evaluate $G_{l_k}$.  Under certain
    integrability assumptions, $G_{l_k}$ is an unbiased estimator of
    $\nabla\Phi(x_k)$,
    i.e.~$\nabla\Phi(x) = \mathbb E_\xi[\nabla\varphi(x,\xi)]$, and it
    holds %
    \begin{equation*}
      \mathbb E[\|\nabla \Phi(x) - G_{l}(x)\|] \le
      \sqrt{\mathbb E[\|\nabla \Phi(x) - G_{l}(x)\|^2]}
      \le {\left(\frac{\mathbb E_\xi[\|\nabla\varphi(x,\xi)-\mathbb E_\xi[\nabla\varphi(x,\xi)]\|^2]}{l}\right)^{1/2},}
    \end{equation*}
    i.e.\ we have the second inequality in \eqref{eq:ass} (up to a
    constant, cp.~Rmk.~\ref{rmk:constant}) with $\alpha =
    1/2$. Applying the multilevel strategy \eqref{eq:ML} then yields a
    stochastic optimization method with increasing batchsize.

    By Corollary \ref{cor:MLSGD}, achieving error
    $\mathbb{E}[\|x_K-x_*\|]\le\eps$ requires $O(\eps^{-1/2})$
    evaluations of $\nabla\varphi(x,\xi)$ (since
    $\sum_{j=0}^{K-1}l_{K,j}(\eps)$ coincides with the number of
    evaluations of $\nabla\varphi(x,\xi)$ in the current
    setting). This is the same convergence rate that is obtained for
    batch size $1$ and decreasing step size $\eta_k\sim \frac{1}{k}$.
    However, we point out that the present multilevel version, which
    uses constant step size and increasing batch size, allows for
    parallelization in the gradient evaluations.
  \end{example}

\subsection{Accelerated Gradient Descent (AGD)}
We write the accelerated gradient descent algorithm as the following
update (see~\cite{Nesterov83})
\begin{align*}
  p_{k+1} &= q_k - \alpha %
            \nabla\Phi
            (q_k)\\
  q_{k+1} &= p_{k+1} + \beta(p_{k+1} - p_k), 
\end{align*}
where $\alpha = 1/L$ and
$\beta = \frac{\sqrt{L} - \sqrt{\mu}}{\sqrt{L}+ \sqrt{\mu}}.$ Defining
$x_k = [p_k, q_k]$ we can represent AGD as~\eqref{eq:disc_seq} using
the operator
\begin{align*}
  \Psi = \left[\begin{array}{cc}
                 0 & (I - \alpha %
                     \nabla\Phi
                     )\\
                 \beta I & (I + \beta)(I - \alpha %
                           \nabla\Phi)
               \end{array}\right].
\end{align*}

AGD using approximated gradients at level $l_k$ can then be
represented as the following three sequence update (see
e.g.~\cite{Nesterov83} for equivalence):
\begin{subequations}\label{eq:AGD}
  \begin{align}
    x_{k} &= \frac{\tau}{1+ \tau} z_k + \frac{1}{1+\tau} y_k\\
    y_{k+1} &= x_{k} - \frac{1}{L}g_{l_k}(x_k)\\
    z_{k+1} &= z_k + \tau(x_k - z_k)- \frac\tau\mu g_{l_k}(x_k).
  \end{align}
\end{subequations}
\begin{proposition}\label{prop:AGD}
  Suppose that $g_l$ satisfies \eqref{eq:ass2} with
  $\eta = 1/\sqrt{L}$ and $\tau = \sqrt{\mu/L}$.  Then $(y_{k},z_{k})$
  generated by \eqref{eq:AGD} satisfies \eqref{eq:disc} with
  $e_k = \Phi(y_k) - \Phi(x_*) + \frac{\mu}{2} \|z_k - x_*\|^2 $
  exponent $2\alpha$, and $c = 1- \tau$, more precisely
  \begin{equation*}
    \Phi(y_{k+1}) - \Phi(x_*) + \frac{\mu}{2} \|z_{k+1} - x_*\|^2
    \le (1-\tau)\Big( \Phi(y_k) - \Phi(x_*) + \frac{\mu}{2} \|z_k - x_*\|^2\Big)+l_k^{-2\alpha}.
  \end{equation*}
\end{proposition}
\begin{proof}
  We begin by using the $L$-smoothness of $\Phi$:
  \begin{align}
    \Phi(y_{k+1}) - \Phi(x_k) &\leq \langle \nabla \Phi(x_k), y_{k+1} - x_k\rangle + \frac{L}{2}\|x_k - y_{k+1}\|^2  \notag\\&= - \frac{1}{L} \langle \nabla \Phi(x_k), g_l(x_k)\rangle + \frac{1}{2L}\|g_l(x_k)\|^2 \notag \\
                              & = \frac{1}{2L}\|g_l(x_k) - \nabla \Phi(x_k) \|^2  - \frac{1}{2L}\|\nabla \Phi(x_k)\|^2 \leq \frac{l_k^{-2\alpha}}{2L\eta^2} - \frac{1}{2L}\|\nabla \Phi(x_k)\|^2.\label{eq:descent}
  \end{align}
  Denote
  $\tilde{z}_{k+1} := z_{k+1} - \frac{\tau}{\mu} (\nabla \Phi(x_k) -
  \nabla g_{l_k}(x_k)) $. The triangle inequality results in the
  following upper bound
  \begin{align*}
    \frac{\sqrt{\mu}}{2}\|z_{k+1} - x_*\| %
    \leq \frac{1}{2\sqrt{L}}\| \nabla \Phi(x_k) - g_l(x_k) \| + \frac{\sqrt{\mu}}{2}\|\tilde{z}_{k+1}-  x_*\| 
    \leq   \frac{l_k^{-\alpha}}{2\sqrt{L} \eta}+ \frac{\sqrt{\mu}}{2}\|\tilde{z}_{k+1}-  x_*\| .
  \end{align*}
  and using Jensen's inequality ($(a/2+b/2)^2 \le a^2/2+b^2/2$) we
  obtain the subsequent identity,
  \begin{align*}
    \frac{\sqrt{\mu}}{2}\|z_{k+1} - x_*\|^2 \leq  \frac{l_k^{-2\alpha}}{2L \eta^2}+ \frac{\mu}{2}\|\tilde{z}_{k+1}-  x_*\|^2.
  \end{align*}
  The rest of the Lyapunov analysis follows as normal (see
  e.g.~\cite{tu2017breaking} and~\cite{wilson2021lyapunov}).  We use
  the previous inequality and strong convexity to obtain the bound:%
  \begin{align*}
    \frac{\mu}{2}&\|z_{k+1} - x_*\|^2 - \frac{\mu}{2}\|z_{k} - x_*\|^2 
                   \leq \frac{\mu}{2}\|\tilde{z}_{k+1} - x_*\|^2 - \frac{\mu}{2}\|z_{k} - x_*\|^2 + \frac{l_k^{-2\alpha}}{2L \eta^2} \\
                 & = \tau \langle \nabla \Phi(x_k), x_* - z_k \rangle  - \tau\mu \langle x_k - z_k , x_* - z_k \rangle +  \frac{\mu}{2}\|\tilde{z}_{k+1} - z_k\|^2  + \frac{l_k^{-2\alpha}}{2L \eta^2} \\
                 & \leq  \tau \langle \nabla \Phi(x_k), x_k - z_k \rangle  - \tau\mu \langle x_k - z_k , x_* - z_k \rangle +  \frac{\mu}{2}\|\tilde{z}_{k+1} - z_k\|^2 \\
                 &\quad - \tau(\Phi(x_k) - \Phi(x_*) + \frac{\mu}{2}\|x_k - x_*\|^2 ) +  \frac{l_k^{-2\alpha}}{2L \eta^2} \\
                 & = \langle \nabla \Phi(x_k), x_k - y_k \rangle  - \frac{\tau\mu}{2} \| x_k - z_k \|^2 +  \frac{\mu}{2}\|\tilde{z}_{k+1} - z_k\|^2 + \ \frac{l_k^{-2\alpha}}{2L \eta^2}\\
                 &\quad - \tau(\Phi(x_k) - \Phi(x_*) + \frac{\mu}{2}\|z_k - x_*\|^2 ).
  \end{align*}
  The inequality uses the strong convexity of $\Phi$. Using the strong
  convexity of $\Phi$ again and the descent bound~\eqref{eq:descent}
  we have:
  \begin{align*}
    e_{k+1}  &\leq (1- \tau) e_k + \Phi(x_k) - \Phi(y_k) + \langle \nabla \Phi(x_k), x_k - y_k\rangle   - \frac{1}{2L}\|\nabla \Phi(x_k)\|^2 +\frac{l_k^{-2\alpha}}{L\eta^2} \\
             &\quad - \tau(\Phi(x_k) - \Phi(y_k))   - \frac{\tau\mu}{2} \| x_k - z_k \|^2 +  \frac{\mu}{2}\|\tilde{z}_{k+1} - z_k\|^2 \\
             & \leq  c e_k  -\frac{\mu}{2}\|x_k - y_k\|^2  - \frac{1}{2L}\|\nabla \Phi(x_k)\|^2 - \tau\langle \nabla \Phi(x_k), x_k - y_k\rangle + \frac{\tau L}{2}\|x_k - y_k\|^2  \\
             &\quad - \frac{\tau\mu}{2} \| x_k - z_k \|^2+  \frac{\mu}{2}\|\tilde{z}_{k+1} - z_k\|^2 +\frac{l_k^{-2\alpha}}{L\eta^2}  \\
             & = c e_k +  (\frac{\tau^2}{\mu}- \frac{1}{2L})\|\nabla \Phi(x_k)\|^2 + (\frac{\tau L}{2} - \frac{\mu}{2\tau})\| x_k - y_k \|^2+l_k^{-2\alpha} \\
             & = c e_k +l_k^{-2\alpha}.
  \end{align*}
  The second line uses strong convexity and smoothness of $\Phi$. The
  following line expands the term
  $\|\tilde{z}_{k+1} - z_k\|^2 = \|y_k - x_k - \frac{\tau}{\mu}\nabla
  \Phi(x_k)\|^2$ and uses the smoothness of $\Phi$ and identity
  $\eta = 1/\sqrt{L}$.
\end{proof}

  \begin{corollary}[MLAGD]\label{cor:MLAGD}
    Consider the setting of Proposition~\ref{prop:AGD}. Then with the
    levels $l_{K,j}(\eps)$ as in \eqref{eq:ML}, $(y_{K},z_{K})$
    generated by \eqref{eq:AGD} satisfies
    $e_K:=\Phi(y_K)-\Phi(x_*)+\frac{\mu}{2}\|z_K-x_*\|^2\le\eps$, and
    it holds $\sum_{j=0}^{K-1}l_{K,j}(\eps)=O(\eps^{-2\alpha})$ as
    $\eps\to 0$.
  \end{corollary}

  Note that the cost $\sum_{j=0}^{K-1} l_{K,j}(\eps)$ for AGD in
  Corollary \ref{cor:MLAGD} increases at twice the rate compared to GD
  in Corollary \ref{cor:MLGD}. However, the AGD result is formulated
  for a quadratic cost quantity, so that the resulting convergence
  rate of the error in terms of the cost is asymptotically the same.

  \subsection{Accelerated Stochastic Gradient Descent (ASGD)}

  We now consider again the stochastic setting, i.e.\ we assume we are
  given random variables $G_{l}$ as in \eqref{eq:ass3}.  In the
  following let
  \begin{subequations}\label{eq:SAGD}
    \begin{align}%
      x_{k} &= \frac{\tau}{1+ \tau} z_k + \frac{1}{1+\tau} y_k \label{eq:upcoup}\\ 
      y_{k+1} &= x_{k} - \frac{1}{L}G_{l_k}(x_k)\\
      z_{k+1} &= z_k + \tau(x_k - z_k)- \frac\tau\mu G_{l_k}(x_k).
    \end{align}
  \end{subequations}
  \begin{proposition}\label{prop:ASGD}
    Suppose that $G_l$ satisfies \eqref{eq:ass3} with
    $\eta = 1/\sqrt{L}$ where $\tau = \sqrt{\mu/L}$.  Then
    $(y_{k},z_{k})$ generated by \eqref{eq:SAGD} satisfies
    \eqref{eq:disc} with
    $e_k = \mathbb{E}[\Phi(y_k)] - \Phi(x_*) + \frac{\mu}{2}
    \mathbb{E}[\|z_k - x_* \|^2]$, exponent $2\alpha$ and
    $c = 1- \tau$, more precisely
    \begin{equation*}
      \mathbb{E}[\Phi(y_{k+1})] - \Phi(x_*) + \frac{\mu}{2} \mathbb{E}[\|z_{k+1} - x_* \|^2]
      \le (1-\tau)\Big( \mathbb{E}[\Phi(y_k)] - \Phi(x_*) + \frac{\mu}{2} \mathbb{E}[\|z_k - x_* \|^2]\Big)+l_k^{-2\alpha}.
    \end{equation*}
  \end{proposition}
  \begin{proof}
    The proof proceeds in the exact same manner as accelerated
    gradient descent replacing $g_{l_k}$ with $G_{l_k}$. %
    In particular note that
    \begin{align*}
      \mathbb{E}_k[\Phi(y_{k+1})] - \Phi(x_k) %
      & \leq \frac{1}{2L}\mathbb{E}_k[\|G_l(x_k) - \nabla \Phi(x_k) \|^2  - \|\nabla \Phi(x_k)\|^2] \leq \frac{l_k^{-2\alpha}}{2L\eta^2} - \frac{1}{2L}\mathbb{E}_k\|\nabla \Phi(x_k)\|^2,
    \end{align*}
    and
    \begin{align*}
      \frac{\sqrt{\mu}}{2}\mathbb{E}_k\|z_{k+1} - x_*\| %
      &\leq \frac{1}{2\sqrt{L}}\mathbb{E}_k\| \nabla \Phi(x_k) - G_l(x_k) \| + \frac{\sqrt{\mu}}{2}\mathbb{E}_k\|\tilde{z}_{k+1}-  x_*\|.  
    \end{align*}
    Therefore,
    \begin{align*}
      \frac{\sqrt{\mu}}{2}\mathbb{E}_k\|z_{k+1} - x_*\|^2 \leq  \frac{l_k^{-2\alpha}}{2L \eta^2}+ \frac{\mu}{2}\mathbb{E}_k\|\tilde{z}_{k+1}-  x_*\|^2.
    \end{align*}
    Given the remainder of the proof relies on the strong convexity
    and smoothness of the function and update~\eqref{eq:upcoup} we
    obtain the following recursion following the same line of
    argumentation:
    \begin{align*}
      e_{k+1} \leq c e_k + l_k^{-2\alpha}.
    \end{align*}

  \end{proof}

  \begin{corollary}[MLASGD]
    Consider the setting of Proposition \ref{prop:ASGD}. Then with the
    levels $l_{K,j}(\eps)$ as in \eqref{eq:ML}, $(y_{K},z_{K})$
    generated by \eqref{eq:SAGD} satisfies
    $e_K:=\mathbb{E}[\Phi(y_k)] - \Phi(x_*) + \frac{\mu}{2}
    \mathbb{E}[\|z_k - x_* \|^2] \le\eps$, and it holds
    $\sum_{j=0}^{K-1}l_{K,j}(\eps)=O(\eps^{-2\alpha})$ as $\eps\to 0$.
  \end{corollary}

\begin{remark}
  The Monte Carlo estimator discussed in the context of SGD, for
  example, can be used for accelerated SGD where we require increasing
  batchsize to ensure assumption~\eqref{eq:ass} holds with
  $\alpha = 1/2$.
\end{remark}

\section{Details for the running example}\label{app:running}
For every $f\in L^2(D)$ on the bounded convex polygonal Lipschitz
domain $D\subseteq\R^2$, denote by $u_f\in H_0^1(D)\subseteq L^2(D)$
the unique weak solution to \eqref{eq:PDE} with right-hand side $f$,
i.e.\
\begin{equation}\label{eq:PDE_weak}
  \int_D\nabla u_f^\top(s) \nabla v(x)+u_f(s)v(s)\dd s=\int_D f(s)v(s)
  \dd s\qquad\forall v\in H_0^1(D).
\end{equation}
\subsection{Well-definedness and regularity of
  $u_f$}\label{app:running_well}
Existence and well-definedness of $u_f$ is classical. Moreover, one
has the apriori estimate $\|u_f\|_{H_0^1(D)}\le \|f\|_{H^{-1}(D)}$, we
refer for example to \cite[\S 26.1]{MR4269305}.

Furthermore convexity of $D$ in fact implies $H^2(D)$ regularity of
$u_f$ and $\|u_f\|_{H^2(D)}\le\|f\|_{L^2(D)}$, see \cite[Theorem
31.30]{MR4269305}.

\subsection{Formula for $\nabla\Phi(f)$}\label{sec:running_gradient}
Consider the objective
\begin{equation}\label{eq:Phifobj}
  \Phi(f)=\frac12\|\Gamma^{-1/2}(F(f)-y)\|_{\R^{n_y}}^2
  +\frac{\lambda}{2}\|f\|_{L^2(D)}^2=:\cV(f,y)+R(f).
\end{equation}
in Example \ref{ex:2}. Clearly the gradient of $R(f)$ equals
$\lambda f\in L^2(D)$. It remains to compute
$\nabla_f\cV(f,y)\in L^2(D)$.

Introduce the operators
\begin{equation*}
  \cS:=\begin{cases}
    \R^{n_y}\to\R\\
    w\mapsto \frac{1}{2}\|\Gamma^{-1/2}(w-y)\|^2_{\R^{n_y}},
  \end{cases}\qquad
  \cO:=\begin{cases}
    L^2(D)\to\R^{n_y}\\
    f\mapsto (\int_D\xi_jf)_{j=1}^{n_y},
  \end{cases}
\end{equation*}
and
\begin{equation*}
  \cA:=\begin{cases}
    L^2(D)\to L^2(D)\\
    h\mapsto u_h,
  \end{cases}
\end{equation*}

Observe that $\cV(f,y)=\cS(\cO(\cA(f)))$.  Using that
$\cA:L^2(D)\to H_0^1(D)$ is bounded linear, it is easily seen that
$\cA:L^2(D)\to L^2(D)$ is bounded linear and self-adjoint.  Therefore
(by the chain rule)
\begin{equation}\label{eq:nablafcV}
  \nabla_f\cV(f,y) =\nabla (\cS\circ\cO\circ\cA)(f)=
  \cA^*(\nabla (\cS\circ \cO)(\cA(f)))
  =\cA (\nabla (\cS\circ \cO)(u_f))\in L^2(D),
\end{equation}
where $u_f=\cA(f)\in H_0^1(D)\subseteq L^2(D)$. We next compute
$\nabla (\cS\circ\cO)$. Denoting by $D\cS(w)$ the Fr\'echet
derivative, it holds
\begin{equation*}
  D\cS(w)(v) = (w-y)^\top\Gamma^{-1}v\in\R\qquad\forall v\in \R^{n_y}
\end{equation*}
and with $\xi=(\xi_j)_{j=1}^{n_y}\in L^2(D,\R^{n_y})$ since $\cO$ is
bounded linear,
\begin{equation*}
  D\cO(f)(g) = \cO(g)=\int_D g(s) \xi(s) \dd s \in \R^{n_y}
  \qquad\forall g\in L^2(D).
\end{equation*}
Hence for the composition
\begin{equation*}
  D(\cS\circ\cO)(f)(g) = D\cS (\cO(f))(D\cO(f)(g))
  =\int_D (\cO(f)-y)^\top
  \Gamma^{-1}\xi(s)
  g(s)\dd s\qquad\forall g\in L^2(D).
\end{equation*}
This shows
\begin{equation*}
  \nabla(\cS\circ\cO)(f) = (\cO(f)-y)^\top
  \Gamma^{-1}\xi(\cdot)\in L^2(D)
\end{equation*}
and finally by \eqref{eq:nablafcV}, $\nabla_f\cV(f,y) = \cA(h)=u_h$
with $h(\cdot)=(\cO(u_f)-y)^\top\Gamma^{-1}\xi(\cdot)$.

    \subsection{Finite element approximation of $\nabla\Phi(f)$}\label{sec:running_gradapp}
    Next we argue that $u_h$ can be approximated with the rate claimed
    in Example \ref{ex:2}. According to, e.g., \cite[\S 26.3.3., \S
    32.3.2]{MR4269305}, given $f\in L^2(D)$, the FEM approximation
    $u_f^l$ to $u(f)=\cA(f)$ on a uniform simplicial mesh on
    $D\subseteq\R^2$ with $O(l)$ elements will satisfy
    $\|u_f-u_f^l\|_{L^2(D)}\lesssim l^{-1}$. %
    Now set
    $\tilde h(\cdot):=(\cO(u_f^l)-y)^\top\Gamma^{-1}\xi(\cdot)\in
    L^2(D)$. Then, since $\cO:L^2(D)\to\R^{n_y}$ is bounded linear and
    $\xi\in L^2(D,\R^{n_y})$, we have
    $\|h-\tilde h\|_{L^2(D)}\lesssim l^{-1}$.  Using that
    $\cA:L^2(D)\to L^2(D)$ is bounded linear,
    $\|u_{h}-u_{\tilde h}\|_{L^2(D)}= \|\cA(h-\tilde
    h)\|_{L^2(D)}\lesssim l^{-1}$.  Finally,
    $\|u_{\tilde h}-u_{\tilde h}^l\|_{L^2(D)}\lesssim l^{-1}$, and
    thus by the triangle inequality
    $\|u_{h}-u_{\tilde h}^l\|\lesssim l^{-1}$ as claimed.

    Then
    \begin{equation}\label{eq:glf}
      g_l(f):=u_{\tilde h}^l+\lambda f
    \end{equation}
    is an approximation to $\nabla\Phi(f)$ satisfying
    $\|g_l(f)-\nabla\Phi(f)\|_{L^2(D)}\lesssim l^{-1}$.

    \subsection{Multilevel convergence of gradient
      descent}\label{app:running_conv}
    We first show $L$-smoothness and $\mu$-strong convexity for the
    objective in \eqref{eq:Phifobj}.
  
    With the notation from Sec.~\ref{sec:running_gradient} the norm of
    the operator $\cO:L^2(D)\to\R^{n_y}$ satisfies
    $\|\cO\|=\|\xi\|_{L^2(D)}$.  Hence for all $f$, $g\in L^2(D)$
    \begin{align*}
      \|\nabla\Phi(f)-\nabla\Phi(g)\|_{L^2(D)}
      &=\|\cO(u_f-u_g)^\top\Gamma^{-1}\xi(\cdot)+\lambda(f-g)\|\nonumber\\
      &\le (\|\xi\|_{L^2(D\R^{n_y})}^2\|\Gamma^{-1}\|_{\R^{n_y\times n_y}}+\lambda)\|f-g\|_{L^2(D)}.
    \end{align*}
    This shows that $\Phi$ is $L$-smooth with
    $L=\|\xi\|_{L^2(D\R^{n_y})}^2\|\Gamma^{-1}\|_{\R^{n_y\times
        n_y}}+\lambda$.

    Moreover, since $F:L^2(D)\to\R^{n_y}$ is bounded linear, the term
    $f\mapsto \frac12\|\Gamma^{-1/2}(F(f)-y)\|_{\R^{n_y}}^2$ is
    convex. Hence, the added regularizer ensures $\Phi$ in
    \eqref{eq:Phifobj} to be $\mu$-strongly convex with $\mu=\lambda$.

    In all this shows that Example \ref{ex:2} is in the setting of
    Example \ref{ex:gdagd} with $\alpha=1$ (i.e.\ $\Phi$ is
    $L$-smooth, $\mu$-strongly convex, and $g_l$ in \eqref{eq:glf}
    satisfies the first inequality in \eqref{eq:ass} with
    $\alpha=1$). Thus, for small enough $\eta>0$, the iteration
    $f_{j+1}=f_j-\eta\nabla g_{l_j(\eps)}(f_k)$ with the approximate
    gradient from subsection \ref{sec:running_gradapp}, converges to
    the unique minimizer $f_*\in L^2(D)$ of $\Phi$ as
    $k\to\infty$. With the multilevel choice $l_{K,j}(\eps)$,
    $j=1,\dots,K(\eps)$ as in \eqref{eq:ML}, it holds
    $\|f_{K(\eps)}-f_*\|_{L^2(D)}\lesssim \eps$ and the cost
    $\sum_{j=0}^{K(\eps)-1}l_{K,j}(\eps)$, which (up to a constant)
    can be interpreted as the computational cost of all required FEM
    approximations, behaves like $O(\eps^{-1})$ as $\eps\to 0$
    according to Theorem \ref{thm:ML}.

    \section{Optimality of the multilevel rate for gradient
      descent}\label{app:opt_rate}
    We give a simple example in the setting of Example \ref{ex:gdagd}
    to show that our results in Sec.~\ref{sec:ML_opti} are sharp in
    general.  For some fixed $\alpha>0$, consider the objective
    function and its approximation
    \[\Phi(x) = \frac12 x^2,\quad \Phi_l(x) =
      \frac12(x-l^{-\alpha})^2,\quad x\in\mathbb R.\] Denote the
    unique minimizer of $\Phi$ by $x_*:=0$.  We consider gradient
    descent with a fixed step size $\eta\in (0,1)$, which amounts to
    (cp.~\eqref{eq:Psi})
    \begin{equation*}
      \Psi(x) = x-\eta x,\qquad\Psi_l(x) = x-\eta(x-l^{-\alpha}).
    \end{equation*}
    Hence for some initial value $x_0\in\R$, \eqref{eq:disc_seq}
    becomes
    \begin{equation}\label{eq:counter}
      x_{k+1} =
      \Psi_{l_k}(x_k)=(1-\eta)x_k+\eta l_k^{-\alpha}
    \end{equation}
    and thus assumption \eqref{eq:disc} holds with
    $c:=1-\eta\in (0,1)$.

\begin{proposition}
  Let $x_0\ge 0$, and let $x_k$ be as in \eqref{eq:counter}.  Then for
  every $\eps>0$, for every $K\in\N$ and for every
  $(l_j)_{j=0}^{K-1}\in (0,\infty)^K$ such that $|x_K-x_*|\le\eps$,
  \[\sum_{j=0}^{K-1} l_j \ge
    \eta^{\frac{1}{\alpha}}\varepsilon^{-\frac1\alpha}.\]
\end{proposition}
\begin{proof}
  We have %
  \begin{equation*}
    x_{K} %
    =c x_{K-1} + (1-c) l_{K-1}^{-\alpha}
    =c^{K} x_0 + \sum_{j=0}^{K-1}(1-c)c^{K-1-j} l_j^{-\alpha}
    \ge (1-c)\sum_{j=0}^{K-1} c^{K-1-j} l_j^{-\alpha}.
  \end{equation*}
  The minimizer of $\sum_{j=0}^{K-1} l_j$ under the constraint
  $(1-c)\sum_{j=0}^{K-1} c^{K-1-j} l_j^{-\alpha}\le \varepsilon$ %
  satisfies according to Lemma \ref{lemma:ML} and \eqref{eq:costML},
  \begin{equation*}
    \sum_{j=0}^{K-1}l_j=\left(\frac{\eps}{1-c}\right)^{-\frac{1}{\alpha}}
    \left(\frac{1-c^{\frac{K}{1+\alpha}}}{1-c^{\frac{1}{1+\alpha}}} \right)^{\frac{1+\alpha}{\alpha}}
    \ge\left(\frac{\eps}{1-c}\right)^{-\frac{1}{\alpha}}
    =\eta^{\frac{1}{\alpha}}\eps^{-\frac{1}{\alpha}}.
  \end{equation*}
\end{proof}

\section{Notation: particle methods}\label{app:notation}
For a particle system $(x^{(m)})_{m=1,\dots,M}$,
$x^{(m)}\in\mathbb R^{n_x}$ and forward operator
$H:\mathbb R^{n_x}\to \mathbb R^{n_z}$, we denote the empirical mean
and covariance operators by
\begin{align*}
  \bar x &= \frac{1}{M}\sum_{m=1}^M x^{(m)},\quad \bar H = \frac1M\sum_{m=1}^M H(x^{(m)}),\\
  C^{x,H}(x) &= \frac1M\sum_{m=1}^M (x^{(m)}-\bar x)\otimes (H(x^{(m)})-\bar H),\\ 
  C^{H,H}(x) &=\frac1M\sum_{m=1}^M (H(x^{(m)})-\bar H)\otimes (H(x^{(m)})-\bar H),\\
  C(x) &= C^{x,x}(x) = \frac1M\sum_{m=1}^M (x^{(m)}-\bar x)\otimes (x^{(m)}-\bar x).
\end{align*}

\section{Algorithm: Multilevel ensemble Kalman
  inversion}\label{app:algo_ML_EKI}

The original EKI method in \cite{ILS2013} has been derived through an
artificial discrete-time data assimilation problem and was formulated
as iterative scheme. For a fixed ensemble size $M$ the time-dynamical
particle system $\{v_j^{(m)}\}_{m=1}^M$ can be written as
\begin{equation}\label{eq:EKI_discrete}
  v_{j+1}^{(m)} = v_j^{(m)} + C^{v,H}(v_j)(C^{H,H}(v_j)+h^{-1}\Sigma)^{-1}(z_{j+1}^{(m)}-H(v_j^{(m)})),\quad j=1,\dots,J,
\end{equation}
where $z_{j+1}^{(m)}\sim\cN(y,h^{-1}\Sigma)$ are perturbed
observations. Viewing $h>0$ as step size the authors in
\cite{SchSt2017} motivated to take the limit $h\to0$ resulting in the
system of coupled SDEs \eqref{eq:EKI_cont}, which has been analysed
rigorously in \cite{BSW2018,BSWW2021}. To be more precise, under weak
assumptions on the general nonlinear forward model $H$, convergence in
probability can be verified, whereas strong convergence can be
verified for linear models. %

We reformulate the update formula \eqref{eq:EKI_discrete} as
\begin{equation}\label{eq:EKI_discrete2}
  v_{j+1}^{(m),l} = v_{j}^{(m),l} + \frac{1}{M}\sum_{r=1}^M \alpha_j^{(r),(m),l} v_j^{(r)}
\end{equation}
verifying the well-known subspace property of the EKI, which will not
be violated for our proposed multilevel formulation. This comes from
the fact, that the updating force formulated in the coordinate system
spanned by the initial ensemble depends on the discretization level
only through the scalar valued coordinates
\[\alpha_j^{(r),(m),l} = \langle H_{l}(v_j^{(r)})-\bar H_{l} ,
  (C^{H_l,H_l}(v_j^{l})+h^{-1}\Sigma)^{-1}(z_{j+1}^{(m)}-H_l(v_t^{(m),l}))\rangle.\]
This observation is useful for an efficient implementation of the
multilevel formulation of EKI and TEKI respectively based on the
discretization scheme \eqref{eq:EKI_discrete} which we summarize as
algorithm below. Note that for standard EKI one needs to run the
algorithm with the choice $H_l\equiv F_l$, $z=y\in\mathbb R^{n_y}$ and
$\Sigma=\Gamma$ and for TEKI, with the corresponding choice
\[H_l(\cdot) = \begin{pmatrix}F_l(\cdot)\\
    {\mathrm{Id}}\end{pmatrix},\quad z = \begin{pmatrix}y\\
    {\bf{0}}_{\mathbb R^{n_x}} \end{pmatrix},\quad \Sigma
  = \begin{pmatrix}\Gamma &0\\ 0 &
    \frac{1}{\lambda}C_0\end{pmatrix}. \]
 
\begin{algorithm}
  \begin{algorithmic}[1]
    \Require initial ensemble $v_0^{(m),0}$, ensemble size $M$,
    forward model $(H_l)_{l\ge0}$, bias parameter $\alpha>0$, rate
    parameter $c\in(0,1)$, step size $h>0$, time-interval length
    $\tau $ such that $N = \tau /h\in\mathbb N$, tolerance
    $\varepsilon>0$.  \State set the number of iteration
    $K\propto \log(\varepsilon^{-1})$ \State set
    $x_0 = \frac{1}{M} \sum\limits_{m=1}^M v_{0}^{(m),0}$ \State
    \textbf{For} {$j=0,\dots,K-1$} \State \quad compute level
    $l_j\propto \varepsilon^{-\frac1\alpha}
    c^{\frac{K-1-j}{1+\alpha}}$ \State \quad \textbf{If} {j=0} \State
    \qquad $v_0^{(m),l_0} = v_{0}^{(m),0},\ m=1,\dots,M$ \State \quad
    \textbf{Else} \State \qquad
    $v_0^{(m),l_j} = v_{\tau }^{(m),l_{j-1}},\ m=1,\dots,M$ \State
    \quad \textbf{EndIf} \State \quad \textbf{For} {$n=0,1,\dots N-1$}
    \State \qquad
    $y_{n+1}^{(m),l_j} \sim\cN(y,h^{-1}\Gamma),\ m=1,\dots,M$ \State
    \qquad
    $\alpha_{n}^{(r),(m),l_j} = \langle F_{l_j}(v_n^{(r)})-\bar
    F_{l_j} ,
    (C^{F_{l_j},F_{l_j}}(v_n^{l_j})+h^{-1}\Gamma)^{-1}(y_{n+1}^{(m),l_j}-F_{l_j}(v_n^{(m),l_j}))\rangle$
    \State \qquad
    $v_{n+1}^{(m),l_j} = v_{n}^{(m),l_j}+\frac{1}{M}\sum_{r=1}^M
    \alpha_n^{(r),(m),l_j} v_n^{(r),l_j}$ \State \quad \textbf{EndFor}
    \State \quad set
    $x_{j+1} = \frac1M\sum\limits_{m=1}^M v_{N}^{(m),l_j}$ \State
    \textbf{EndFor}
  \end{algorithmic}
  \caption{Multilevel ensemble Kalman inversion
    (ML-EKI)}\label{alg:MLEKI}
\end{algorithm}

\section{Proofs of Section \ref{sec:MLEKI}}
\subsection{Proof of
  Proposition~\ref{prop:error_EKI}}\label{app:proof_Prop_EKI}
\begin{proof}
  We first fix an iteration $j$ with level $l_j$ and suppress the
  dependence of $F_{l_j}$ on $l_j$. The evolution equation of the
  particle system for standard EKI can be written as
  \[\mathrm{d}v_t^{(m)} = -(C(v_t)+B)F^\top
    \Gamma^{-1}(Fv_t^{(m)}-y)\,{\mathrm d}t + C(v_t)F^\top
    \Gamma^{-1/2}\,{\mathrm d}W_t^{(m)}.\] We define
  $\bar{\mathfrak r}_t:= \Gamma^{-1/2}F(\bar v_t-v^\dagger)$ and
  $\mathfrak r_t^{(m)}:= \Gamma^{-1/2}F(v_t^{(m)}-v^\dagger)$ such
  that $\frac12\|\bar r_t\|^2 = \cV (\bar v_t)$ and
  \[{\mathrm{d}}\bar{\mathfrak r}_t = -(C(\mathfrak r_t)+\tilde
    B)\bar{\mathfrak r}_t\,{\mathrm d}t + C(\mathfrak r_t)\,{\mathrm
      d}\bar W_t \] with
  $\tilde B := \Gamma^{-1/2}FBF^\top \Gamma^{-1/2}$.  Following
  Theorem~5.2 in \cite{BSWW19} we can bound
  \[\mathbb E[\|\bar{\mathfrak r}_{s+\tau }\|^2]\le \mathbb
    E[\|\bar{\mathfrak r}_{s}\|^2]- \sigma \int_s^{s+\tau }\mathbb
    E[\|\bar{\mathfrak r}_{u}\|^2]\,{\mathrm d}u,\] where $\sigma>0$
  denotes the smallest eigenvalue of $\tilde B$. Hence with
  $\bar v_{t_j}^{l_j} := \frac{1}{M}\sum_{m=1}^Mv_{t_j}^{(m),l_j}$ it
  follows
  \[\mathbb E[ \frac{1}{2} \|F_{l_j} \bar
    v_{t_{j+1}}^{l_j}-y\|_\Gamma^2] \le (1-\sigma\cdot \tau)\mathbb E[
    \frac{1}{2} \|F_{l_{j}} \bar v_{t_{j}}^{l_{j-1}}-y\|_\Gamma^2], \]
  where we have defined
  $\|\cdot \|_\Gamma^2 = \|\Gamma^{-1/2}\cdot\|_{\R^{n_y}}^2$. With
  Assumption~\ref{ass:forwardmodel} and the reverse triangle
  inequality $|\|Fx\|-\|Fy\||\le \|Fx-Fy\|$ we have that
  \begin{align*}
    e_{j+1} = \mathbb E[ \frac{1}{2} \|F \bar v_{t_{j+1}}^{l_j}-y\|_\Gamma^2] &= \mathbb E[ \frac{1}{2} \|F_{l_j} \bar v_{t_{j+1}}^{l_j}-y\|_\Gamma^2] + \mathbb E\left[ \frac{1}{2} \|F \bar v_{t_{j+1}}^{l_j}-y\|_\Gamma^2-\frac{1}{2} \|F_{l_j} \bar v_{t_{j+1}}^{l_j}-y\|_\Gamma^2\right]\\
                                                                              &\le (1-\sigma\cdot \tau )\mathbb E[ \frac{1}{2} \|F_{l_{j}} \bar v_{t_{j}}^{l_{j-1}}-y\|_\Gamma^2] + b_1 l_j^{-\alpha}\\
                                                                              &\le (1-\sigma\cdot \tau )\mathbb E[ \frac{1}{2} \|F \bar v_{t_{j}}^{l_{j-1}}-y\|_\Gamma^2]+ (1-\sigma\cdot \tau )b_1 l_j^{-\alpha}+b_1 l_j^{-\alpha}\\
                                                                              &\le (1-\sigma\cdot \tau )e_j + bl_j^{-\alpha},
  \end{align*}
  for some constant $b_1>0$ and $b =(2-\sigma\cdot \tau )b_1 $. We
  note that we have used that
  $\mathbb E[\|\bar v_{t_{j+1}}^{l_j}\|^2]$ remains uniformly bounded,
  see e.g.~Lemma 5 in \cite{DL2021_a}. Moreover, since we assumed
  finite second moments of the initial distribution $Q_0$, we have
  that $\mathbb E[\|\bar v_0\|^2]<\infty$ and by local Lipschitz
  continuity of $x\mapsto\frac12\|Fx-y\|_\Gamma^2$ we have that
  $e_0=\mathbb E[\|F\bar v_0 - y\|_\Gamma^2]<\infty$.  With the above
  computations we have verified that the error quantity $e$ satisfies
  the decay assumption \eqref{eq:disc} (respectively the
  generalization in Rmk.~\ref{rmk:constant}) and therefore, the
  assertion follows by application of Theorem \ref{thm:SL} and Theorem
  \ref{thm:ML}.
\end{proof}

\subsection{Proof of
  Proposition~\ref{prop:error_TEKI}}\label{app:proof_Prop_TEKI}
\begin{proof}
  We again fix an iteration $j$ with level $l_j$ and suppress the
  dependence of $F_{l_j}$ on $l_j$.  The evolution equation of the
  particle system for TEKI can be written as
  \begin{align*}
    \mathrm{d}v_t^{(m)} &= -(C(v_t)+B)(F^\top \Gamma^{-1}(Fv_t^{(m)}-y)+\lambda C_0^{-1}v_t^{(m)})\,{\mathrm d}t\\
                        &\quad + C(v_t)F^\top \Gamma^{-1/2}\,{\mathrm d}W_t^{(m)}+\sqrt{\lambda}C(v_t)C_0^{-1/2}\,{\mathrm d}W_t^{(m)}.
  \end{align*}
  Next, we define
  $\bar{\mathfrak r}_t:= \Sigma^{-1/2}H(\bar v_t- x_\ast)$ and
  $\mathfrak r_t^{(m)}:= \Sigma^{-1/2} H(\bar v_t-x_\ast)$, where
  $x_\ast\in\mathcal X$ is the unique minimizer of $\cV_R$ and $H$,
  $\Sigma$ are defined in \eqref{eq:TEKIdata}, such that
  $\frac12\|\bar r_t\|^2 = \frac12\|F(\bar
  v_t-x_\ast)\|_\Gamma^2+\frac{\lambda}{2} \|\bar
  v_t-x_\ast\|_{C_0}^2$. Since $x_\ast\in\mathcal X$ is the unique
  minimizer of $\cV_R$ it holds true that
  \[0 = \nabla_x \cV_R(x_\ast) = F^\top \Gamma^{-1}(Fx_\ast-y)+\lambda
    C_0^{-1}x_\ast = H^\top \Sigma^{-1}(Hx_\ast - z),\] and with
  $\tilde B:=\Sigma^{-1/2}HBH^\top \Sigma^{-1/2}$ we can write
  \[{\mathrm{d}}\bar{\mathfrak r}_t = -(C(\mathfrak r_t)+\tilde
    B)\bar{\mathfrak r}_t\,{\mathrm d}t + C(\mathfrak r_t)\,{\mathrm
      d}\bar W_t.\] The assertion follows similarly to the proof of
  Proposition~\ref{prop:error_EKI}.
\end{proof}

\section{Algorithm: Multilevel interacting Langevin
  MCMC}\label{app:algo_ML_Langevin}

The multilevel interacting Langevin sampler is based on its particle
approximation \eqref{eq:interacting_Langevin_integral}.  Due to the
finite number of ensemble size $M$, the resulting algorithm contains
an additional empirical error according to the mean-field limit
represented by the Fokker--Planck equation \eqref{eq:FokkerPlanck}. We
refer to \cite{DL2021_b} for a detailed analysis of large ensemble
size limit. We are going to solve these systems of coupled SDEs by a
forward Euler-Maruyama method and emphasize that other numerical
approximation schemes for SDEs can be applied as well. The resulting
multilevel sampling algorithm is summarized in below.

\begin{algorithm}
  \begin{algorithmic}[1]
    \Require initial distribution $q_0$, ensemble size $M$, gradient
    of log-likelihood $(\nabla\Phi_R^l)_{l\ge0}$, bias parameter
    $\alpha>0$, rate parameter $c\in(0,1)$, step size $h>0$,
    time-interval length $\tau $ such that $N = \tau /h\in\mathbb N$,
    tolerance $\varepsilon>0$.  \State set the number of iteration
    $K\propto \log(\varepsilon^{-1})$ \State Initialize particle
    system as i.i.d. sample $v_0^{(m),0}\sim q_0$ \State set
    $\rho_0 = \frac{1}{M} \sum\limits_{m=1}^M \delta_{v_{0}^{(m),0}}$
    \State \textbf{For} {$j=0,\dots,K-1$} \State \quad compute level
    $l_j\propto \varepsilon^{-\frac1\alpha}
    c^{\frac{K-1-j}{1+\alpha}}$ \State \quad \textbf{If} {j=0} \State
    \qquad $v_0^{(m),l_0} = v_{0}^{(m),0},\ m=1,\dots,M$ \State
    \quad\textbf{Else} \State \qquad
    $v_0^{(m),l_j} = v_{\tau }^{(m),l_{j-1}},\ m=1,\dots,M$ \State
    \quad \textbf{EndIf} \State \quad \textbf{For} {$n=0,1,\dots N-1$}
    \State \qquad $\xi_{n+1}^{(m),l_j}\sim\cN(0,I)$ \State \qquad
    $\Delta W_{n+1}^{(m),l_j} = \sqrt{h}\xi_{k+1}^{(m),l_j}$, \State
    \qquad
    $g_n^{(m),l_j} = -h
    C(v_n^{l_j})\nabla\cV_R^{l_j}(v_n^{(m),l_j})+\sqrt{2C(v_t^{l_j})}\Delta
    W_{n+1}^{(m),l_j}$ \State \qquad
    $v_{n+1}^{(m),l_j} = v_{n}^{(m),l_j}+g_n^{(m),l_j}$ \State \quad
    \textbf{EndFor} \State \quad set
    $\rho_{j+1} = \frac1M\sum\limits_{m=1}^M \delta_{v_{N}^{(m),l_j}}$
    \State \textbf{EndFor}
  \end{algorithmic}
  \caption{Multilevel interacting Langevin sampler
    (ML-ILS)}\label{alg:MLLangevin}
\end{algorithm}

\section{Proofs of Sec.~\ref{sec:MLsampling}}
\subsection{Proof of
  Proposition~\ref{prop:error_EKS}}\label{app:proof_Prop_Langevin}
\begin{proof}
  We define $\Phi_l:\cP\to\R$ by
  $\Phi_l(\rho) = \KL{\rho}{\rho_\ast^{l}}$ and $\Phi:\cP\to\R$ by
  $\Phi(\rho) = \KL{\rho}{\rho_\ast}$ and assuming that for
  $\sigma_1,\sigma_2>0$ we have
  $\nabla^2 \cV_R^l>\sigma_1{\mathrm{Id}}$ and
  $C(\rho_j)>\sigma_2{\mathrm{Id}}$. From
  \cite[Proposition~2]{GHWS2020} it follows that there exists a
  constant $c\in(0,1)$ such that
  $\KL{\rho_{j+1}}{\rho_\ast^{l_j}}\le
  \KL{\rho_{j}}{\rho_\ast^{l_j}}$.  Furthermore, under
  Assumption~\ref{ass:KL} it holds true that
  \[|\Phi(\rho_j)-\Phi_l(\rho_j)|\le b l^{-\alpha}\] since by
  definition of the KL divergence we have that
  \begin{align*}
    |\Phi_l(\rho)-\Phi(\rho)| &= |\KL{\rho}{\rho_\ast^l}-\KL{\rho}{\rho_\ast}|\\ &= |\int_{\R^{n_x}} \rho(x) \log(\rho(x))\,{\mathrm d}x - \int_{\R^{n_x}} \rho(x) \log(\rho_\ast^l(x))\,{\mathrm{d}}x\\ &\quad -\int_{\R^{n_x}} \rho(x) \log(\rho(x))\,{\mathrm d}x + \int_{\R^{n_x}} \rho(x) \log(\rho_\ast(x))\,{\mathrm d}x|\\
                              &\le \int_{\R^{n_x}} \rho(x) |\cV_R^l(x)-\cV_R(x)|\,{\mathrm d}x.
  \end{align*}
  With Assumption~\ref{ass:KL} it follows that
  \begin{align*}
    |\Phi_{l_j}(\rho_j)-\Phi(\rho_j)|&\le \int_{\R^{n_x}} \rho_j(x) |\cV_R^{l_j}(x)-\cV_R(x)|\,{\mathrm d}x\\ &\le b \int_{\mathbb R^{n_x}}\rho_j(x)\|F(x)-F_l(x)\|^2\,{\mathrm d}x \le bl^{-\alpha}.
  \end{align*}
  Finally, we obtain%
  \begin{align*}
    \KL{\rho_{j+1}}{\rho_\ast} &= \KL{\rho_{j+1}}{\rho_\ast^{l_j}}+\left(\KL{\rho_{j+1}}{\rho_\ast}-\KL{\rho_{j+1}}{\rho_\ast^{l_j}}\right)\\
                               &\le c \KL{\rho_{j}}{\rho_\ast^{l_j}}+l_j^{-\alpha}\\
                               &\le c \KL{\rho_{j}}{\rho_\ast}+(1+c)l_j^{-\alpha}.
  \end{align*}
  With the above computations we have verified that the error quantity
  $e$ satisfies the decay assumption \eqref{eq:disc} (respectively the
  generalization in Rmk.~\ref{rmk:constant}) and therefore, the
  assertion follows by application of Theorem \ref{thm:SL} and Theorem
  \ref{thm:ML}.
\end{proof}

\section{Details for Sec.~\ref{sec:numerics}}\label{app:numerics}

The numerical approximations $F_l$ to $F$ are computed as follows:
Given $x\in\mathbb R^{n_x}$, we let
$F_l:\mathbb R^{n_x}\to\mathbb R^{n_y}$ be the map defined by
$F_l(f)=\cO(u^l_f)$, where $u^l_f$ denotes the finite element solution
to \eqref{eq:PDE} using continuous piecewise linear finite elements on
a uniform mesh with meshwidth %
$2^{-\tau(l)}$ where $\tau(l) = \ceil{\log(l)/\log(2)}$ such that
$l=2^{\tau(l)}$.  Since $\cO:H_0^1(D)\to R^{n_y}$ is continuous, it
can be shown that $\|F(x)-F_l(x)\|\lesssim l^{-1}\|x\|$ for all $l$
and all $x\in\R^{n_x}$, i.e.\ the convergence rate $\alpha$ in Section
\ref{sec:ML_opti} equals $1$. As a prior on the parameter space
$\R^{n_x}$ we chose $Q_0=\cN(0,C_0)$ with
$C_0 = {\mathrm{diag}}(i^{-2\beta},i=1,\dots,n_x)$ for some fixed
$\beta>0$.

The truth $x^\dagger$ %
was generated as a draw from the prior.  Figure
\ref{fig:complexity_comparison} shows the error convergence of
multilevel and single-level TEKI in Algorithm \ref{alg:MLEKI} in
Appendix \ref{app:algo_ML_EKI} with the parameters $\beta = 1$ and
$n_x = 100$.  The ensemble size was size $M=50$ and the step size of
the discretization scheme \eqref{eq:EKI_discrete} was chosen a
$h=0.1$. %
The plotted error quantity is
$$e_{K(\eps)} = \mathbb E\left[\frac12\|
  F_{{\mathrm{ref}}}(x_{K}(\varepsilon)-x_\ast)\|_\Gamma^2 +
  \frac{\lambda}{2}\|x_K(\varepsilon)-x_\ast\|_{C_0}^2\right],$$ with
the reference solution $x_*$ computed via
$$x_\ast= (F_{{\mathrm{ref}}}^\top\Gamma^{-1} {F_{{\mathrm{ref}}}} + \lambda C_0^{-1})^{-1}{F_{{\mathrm{ref}}}}^\top\Gamma^{-1}y.$$ The cost quantity was
computed as in \eqref{eq:cost}.

The second plot in Figure \ref{fig:complexity_comparison} shows the
convergence of the posterior mean for the multilevel interacting
Langevin sampler (ILS) in Algorithm~\ref{alg:MLLangevin} in Appendix
\ref{app:algo_ML_Langevin}.  In this case we chose $M=2000$ particles,
$\tau =0.1$ and the step size $h=0.001$ of the Euler-Maruyama scheme.
The plotted error quantity shows
$$\mathbb E\left[\frac12\|
f(\cdot, x_{K(\varepsilon)})-f(\cdot,x_\ast)\|_{L^2(D)}^2\right],$$ where
$$x_\ast= ({F_{{\mathrm{ref}}}}^\top\Gamma^{-1} {F_{{\mathrm{ref}}}}+ C_0^{-1})^{-1}
{F_{{\mathrm{ref}}}}^\top\Gamma^{-1}y,$$ which is the posterior mean
on reference accuracy level $2^{14}$, and coincides with the Tikhonov
regularized solution. We see a similar complexity gains as for
multilevel TEKI.

\end{document}